%% file: surgeries_and_QHB.tex
\documentclass [12pt]{amsart}

\usepackage[english]{babel}
\usepackage{amssymb,amsfonts,amsmath,amsthm}
\usepackage[lite]{amsrefs}
\usepackage{epsfig} 
\usepackage{mathrsfs} 
\usepackage{color}
\usepackage{amscd}
\usepackage[all]{xy}
\usepackage{pinlabel}
\usepackage{tikz}
\usepackage{hyperref}

\usepackage[margin=3cm, marginpar=2.5cm]{geometry}

\usepackage[nameinlink,capitalise,noabbrev]{cleveref} 

\newtheorem{lemma}{Lemma}[section]
\newtheorem{thm}[lemma]{Theorem}
\newtheorem{prop}[lemma]{Proposition}
\newtheorem{cor}[lemma]{Corollary}

\theoremstyle{definition}
\newtheorem{defn}[lemma]{Definition}

\newtheorem{quest}[lemma]{Question}

\newtheorem{ex}[lemma]{Example}

\theoremstyle{remark}
 
\newtheorem{rem}[lemma]{Remark} 

\newcommand{\Q}{\mathbb{Q}}
\newcommand{\Z}{\mathbb{Z}}

\newcommand{\spinc}{spin$^c$ }
\newcommand{\Spinc}{{\rm Spin}^c}
\newcommand{\HFp}{\mathrm{HF}^+}
\newcommand{\HFpred}{\mathrm{HF}^+_{\rm red}}
\newcommand{\Tp}{\mathcal{T}^+}

\newcommand{\ft}{\mathfrak{t}}
\newcommand{\fs}{\mathfrak{s}}
\newcommand{\Yt}{(Y,\ft)}
\newcommand{\unknot}{\mathcal{O}}
\newcommand{\eps}{\varepsilon}

\title{Dehn surgeries and rational homology balls}
\author{Paolo Aceto and Marco Golla}

\address{Alfr\'ed R\'enyi Institute of Mathematics, Budapest, Hungary}
\email{aceto.paolo@renyi.mta.hu}

\address{Department of Mathematics, Uppsala University, Uppsala, Sweden}
\email{marco.golla@math.uu.se}

\begin{document}

\begin{abstract}
We consider the question of which Dehn surgeries along a given knot bound rational homology balls.
We use Ozsv\'ath and Szab\'o's correction terms in Heegaard Floer homology to obtain 
general constraints on the surgery coefficients. 
We then turn our attention to the case of integral surgeries, with particular emphasis on positive torus knots.
Finally, combining these results with a lattice-theoretic obstruction based on Donaldon's theorem, we classify which integral surgeries along torus knots of the form $T_{kq\pm 1,q}$ bound rational homology balls.
\end{abstract}

\maketitle

\input{intro.tex}
\input{plumb.tex}
\input{d.tex}
\input{rational.tex}
\input{integral.tex}
\input{torus.tex}

\bibliographystyle{amsplain}
\bibliography{surgeries_and_QHB}

\end{document}

%% file: intro.tex
\section{Introduction}

A \emph{filling} of a closed 3--manifold $Y$ is a smooth, compact 4--manifold $Z$ with $\partial Z = Y$. 
It is very natural to ask how simple fillings of a given $Y$ can be: Lickorish \cite{Lickorish} and Wallace \cite{Wallace} 
showed that there is always a simply-connected filling; 
Milnor \cite{Milnor-spin} proved that there is always a spin filling. 
The Rokhlin invariant \cite{Rokhlin} and, more recently, Donaldson's theorems \cite{Donaldson} provided formidable 
obstructions to the existence of fillings which are \emph{integral} 
or \emph{rational homology balls} (i.e. have the same homology of a point, with integral or rational coefficients).

On the more constructive side, Casson and Harer \cite{CassonHarer} built families of Seifert fibred spaces admitting a rational homology ball filling. 
Lisca \cites{Lisca-ribbon, Lisca-sums} used Donaldson's theorem to find all relations among lens spaces in the rational homology 3--dimensional cobordism group. 
Lecuona \cite{Lecuona} tackled the problem of determining which 3--legged Seifert fibred spaces bound rational homology balls; the first author \cite{Paolo} studied which rational homology $S^1\times S^2$'s bound rational homology $S^1\times D^3$'s.

In another direction,
Owens and Strle \cite{OwensStrle} addressed the question of when the intersection form of $Z$ can be negative definite, 
in the case when $Y$ is obtained as Dehn surgery along a knot in $S^3$, 
and answered it completely in the case of torus knots.

Given a knot $K$ in $S^3$ we address the following question: which Dehn surgeries along $K$ bound rational homology balls? 
From this viewpoint Lisca's work \cite{Lisca-ribbon} can be regarded as a complete answer for the unknot.
Moreover, since Dehn surgeries on concordant knots are homology cobordant, this provides
a complete answer for every smoothly slice knot.

We use Ozsv\'ath and Szab\'o's \emph{correction terms} in Heegaard Floer homology \cite{OzsvathSzabo-absolutely} to obtain constraints on the surgery coefficients.
For instance, we prove the following.

\begin{thm}\label{t:pqfinite}
For every knot $K$ and every positive integer $q$, there exist finitely many positive integers $p$ such that $S^3_{p/q}(K)$ 
bounds a rational homology ball.
\end{thm}

Many of our results will be expressed in terms of the knot invariant $\nu^+$, a byproduct of the Heegaard Floer package defined by Hom and Wu \cite{HomWu} (see \cref{s:prelim} below).

\begin{thm}\label{p:mvsnu1}
There are at most two positive integer values of $n$ such that $S^3_n(K)$ bounds a rational homology ball; 
if there are two, they are consecutive squares and $1+8\nu^+(K)$ is a perfect square.
\end{thm}

We then turn our attention to the case of positive torus knots. 
We prove the following:

\begin{thm}\label{t:pvs9q}
If $p,q > 1$ are coprime integers with $p>9q$, then no positive integral surgery along $T_{p,q}$ bounds a rational homology ball.
\end{thm}

Finally, combining these results with a lattice-theoretic
obstruction based on Donaldson's theorem, 
we classify which integral surgeries along torus knots of the form $T_{kq\pm 1,q}$ bound rational homology balls.

\begin{thm}\label{t:kq+1}
Let $p,q > 1$ be integers such that $p\equiv \pm 1\pmod q$. 
The $3$--manifold $S^3_{n}(T_{p,q})$ bounds a rational homology ball if and only if the triple $(p,q;n)$ 
is one of the following:
\begin{itemize}
 \item[{(i)}] $(p,q;n)=(q+1,q;q^2)$ for some $q\ge 2$;
 \item[{(ii)}] $(p,q;n)=(q+1,q;(q+1)^2)$ for some $q\ge 2$;
 \item[{(iii)}] $(p,q;n)=(4q\pm 1,q;(2q)^2)$ for some $q\ge 2$;
 \item[{(iv)}] $(p,q;n)\in \{(5,2;9),(5,3;16),(13,2;25)\}$;
 \item[{(v)}] $(p,q;n) \in\{(9,4;36), (25,4;100)\}$;
 \item[{(vi)}] $(p,q;n) \in \{(17,3;49),(22,3;64),(43,6;256)\}$.
\end{itemize}
The manifolds corresponding to triples in \emph{(iv)} are lens spaces; the ones corresponding 
to triples in \emph{(v)} are connected sums of two lens spaces; all the other manifolds are Seifert fibred over $S^2$ with three singular fibres.
\end{thm}

It is very interesting to compare our list with the list of \cite{55letters}, where Fern\'andez de Bobadilla, Luengo, Melle Hern\'andez and N\'emethi classify singularities with one Puiseux pair that appear as cusps of rational unicuspidal curves in the complex projective plane. 
The complement of a regular neighbourhood of such a curve (taken with the opposite orientation with respect to the one inherited by $\mathbb{CP}^2$) is in fact a rational homology ball whose boundary is a positive surgery along a torus knot (see also \cite{BorodzikLivingston}).

The family (ii) and the sub-family $(4q-1,q;(2q)^2)$ above correspond to the families (a) and (b) in~\cite[Theorem 1.1]{55letters}; the cases $(2,13;25)$, $(4,25;100)$, belong to the families (d) and (c), respectively, while $(3,22;64)$, and $(6,43;256)$ correspond to cases (e) and (f).
On one hand, the comparison shows that the topological setting is, perhaps unsurprisingly, richer than the algebraic setting; on the other, it shows that the two non-algebraic infinite families we find are not-too-far from the algebraic ones.
Finally, it also shows that there are many more triples $(p,q;n)$ for which $S^3_n(T_{p,q})$ bounds a rational homology ball, if we drop the assumption $p\equiv \pm 1\pmod q$ (namely, the families (c) and (d) in~\cite[Theorem 1.1]{55letters}).

In a forthcoming paper, joint with Kyle Larson, we will address related questions for integral surgeries along cables of knots.

\subsection*{Organisation of the paper} In \cref{s:plumb} and \cref{s:prelim} we recall some basic facts about plumbings and Donaldson's theorem, and about correction 
terms in Heegaard Floer homology.
In \cref{s:rational} we study correction terms of lens spaces and rational surgeries along knots in $S^3$ and prove \cref{t:pqbounds}, which is a quantitative version of \cref{t:pqfinite}.
In \cref{s:integral} we focus on integral surgeries and prove \cref{p:mvsnu1} as a corollary of \cref{p:mvsnu}; we then turn to the special cases of alternating and torus knots, proving \cref{t:pvs9q}. 
Finally, in \cref{torus} we study in more detail surgeries along torus knots and we prove \cref{t:kq+1}.

\subsection*{Acknowledgements} We would like to thank J\'ozsef Bodn\'ar, Daniele Celoria, Kyle Larson, Francesco Lin, Paolo Lisca, 
and Duncan McCoy for interesting conversations. 
The first author was partially supported by the ERC Advanced Grant LDTBud. The second author was partially supported 
by the PRIN--MIUR research project 2010--11 ``Variet\`a reali e complesse: geometria, topologia e analisi armonica'', by the FIRB research project ``Topologia e geometria di variet\`a in bassa dimensione'', and by the Alice and Knut Wallenberg foundation.

%% file: plumb.tex
\section{Plumbed manifolds and rational homology cobordisms}\label{s:plumb}

In this section we briefly recall the language of plumbed manifolds; we then state the lattice-theoretic obstruction based on Donaldson's diagonalisation theorem which we will use in Section \ref{torus}
, and we study the extension of \spinc structures from the boundary to a rational homology ball.

\subsection{Plumbings}

In this paper, a \emph{plumbing graph} $\Gamma$ is a finite \emph{tree} where every vertex has an integral  weight assigned to it.
To every plumbing graph $\Gamma$ we associate a smooth oriented $4$--manifold 
$P(\Gamma)$ with boundary $\partial P(\Gamma)$ in the following way. 
For each vertex take a disc bundle over the 2--sphere with Euler number prescribed by the weight of the vertex. 
Whenever two vertices are connected by an edge we identify the trivial bundles over two small discs (one in each sphere) by exchanging the role of the fiber and the base coordinates. We call $P(\Gamma)$ (resp. $\partial P(\Gamma)$) a \emph{plumbed $4$--manifold} (resp. \emph{plumbed $3$--manifold}).

This definition can be extended to reducible $3$--manifolds; if the graph is a finite forest (i.e. a finite disjoint union of trees) we take the boundary connected sum of the plumbed $4$--manifolds associated to each connected component of $\Gamma$.

The group $H_2(P(\Gamma);\mathbb{Z})$ is free Abelian, generated by the zero sections of the sphere bundles (i.e. by vertices of the graph). 
Moreover, with respect to this basis, the intersection form of $P(\Gamma)$, which we indicate by $Q_{\Gamma}$, is described by the matrix $M_{\Gamma}$ whose entries $(a_{ij})$ are defined as follows:
\begin{itemize}
 \item $a_{ii}$ equals the Euler number of the corresponding disc bundle,
 \item $a_{ij}=1$ if the corresponding vertices are connected,
 \item $a_{ij}=0$ otherwise.
\end{itemize}
Finally note that $M_{\Gamma}$ is also a presentation matrix for the group
$H_1(\partial P(\Gamma);\mathbb{Z})$.

Recall from \cite{Neumann} that every plumbed $3$--manifold has a unique description via a \emph{positive canonical plumbing graph} as well as a \emph{negative} one.
Note that if $\Gamma$ is the positive canonical plumbing graph of a plumbed $3$--manifold, $Q_\Gamma$ is \emph{not} necessarily positive definite.
Moreover, given a plumbed $3$--manifold $\partial P(\Gamma)$ where $\Gamma$ is a positive canonical plumbing graph, there is an algorithm to construct the positive plumbing graph of $-\partial P(\Gamma)$ (i.e. $\partial P(\Gamma)$ with the reversed orientation).
We call this graph \emph{the dual of $\Gamma$} and we denote it with $\Gamma^*$.

\subsection{Intersection forms and \spinc structures}

An \emph{integral lattice} is a pair $(G,Q_G)$ where $G$ is a free Abelian group, equipped with a symmetric bilinear form $Q_G$.
We denote with $(\mathbb{Z}^N,\pm I)$ the standard positive or negative definite diagonal integral lattice.
A morphism of integral lattices is a homomorphism of Abelian groups which preserves the intersection form.
In particular, to a plumbed $3$--manifold we can associate the intersection lattice $H_2(P(\Gamma);\mathbb{Z})$.
 
The next result is implicit in several papers, see for instance~\cite{Lisca-ribbon} and~\cite{Lecuona}.

\begin{prop}\label{p:Donaldson}
Let $\partial P(\Gamma)$ be a rational homology sphere plumbed $3$--manifold described by its positive (respectively negative) canonical plumbing graph $\Gamma$ which has $N$ vertices. 
Suppose that $Q_{\Gamma}$ is definite and that $\partial P(\Gamma)$ bounds a rational homology ball. Then, there exists an injective morphism of integral lattices
\[
 (H_2(P(\Gamma);\mathbb{Z}),Q_{\Gamma})\hookrightarrow (\mathbb{Z}^N,I) 
\]
(resp. $(H_2(P(\Gamma);\mathbb{Z}),Q_{\Gamma})\hookrightarrow (\mathbb{Z}^N,-I)$).
\end{prop}

If there exists an embedding as in the statement above, we will often simply say that $H_2(P(\Gamma);\Z)$ (or even $Q_\Gamma$ or $\Gamma$) embeds.

\begin{proof}
Let $W$ be the rational homology ball bounded by $\partial P(\Gamma)$ and let $X=P(\Gamma)\cup -W$.
By Donaldson's diagonalisation theorem the intersection form on $X$ is standard and the inclusion $P(\Gamma)\hookrightarrow X$ induces the desired morphism of integral lattices. 
\end{proof}

We now turn to studying the extension problem for \spinc structures.
Recall that the space $\Spinc(Z)$ of \spinc structures on a manifold $Z$ is freely acted upon by $H^2(Z;\Z)$, and that the restriction $\Spinc(Z)\to\Spinc(Y)$ induced by an inclusion $Y\hookrightarrow Z$ is equivariant with respect to this $H^2$--action.
Namely, fix a \spinc structure $\fs$ on $Z$ and a class $\alpha\in H^2(Z;\Z)$;
call $\ft$ the restriction of $\fs$ to $Y$ and $\beta\in H^2(Y;\Z)$ the restriction of $\alpha$.
Then, the restriction of $\alpha\cdot\fs$ is $\beta\cdot\ft$.

The following proposition is well-known, and it will be extensively used in the following sections. For completeness, we also sketch a proof.

\begin{prop}\label{p:extension}
Let $K$ be a knot in $S^3$, $r = p/q \neq 0$ be rational, and suppose that $Y=S^3_r(K)$ bounds a rational homology ball $Z$. Then $p = m^2$ is a square and the set of \spinc structures on $Y$ that extend to $Z$ is a cyclic quotient of $H^2(Z;\Z)$ comprised of $m$ elements.
\end{prop}

\begin{proof}[Sketch of proof]
Since $Y$ is a rational homology sphere, $H^1(Y;\Z) = 0$. Moreover, since $\partial Z = Y$, the map $H^3(Y;\Z)\to H^4(Z,Y;\Z)$ is an isomorphism.
The long exact sequence in cohomology for the pair $(Z,Y)$ reads as follows:
\[
0 \to H^2(Z,Y;\Z) \to H^2(Z;\Z) \to H^2(Y;\Z) \to H^3(Z,Y;\Z) \to H^3(Z;\Z) \to 0.
\]
From Poincar\'e--Lefschetz duality and the universal coefficient theorem, one obtains that all groups involved are finite, and moreover $|H^2(Z,Y;\Z)| = |H^3(Z;\Z)|$ and $|H^2(Z;\Z)|=|H^3(Z,Y;\Z)|$.
This shows that $|H^2(Y)| = p$ is a square, and that the image of the map $H^2(Z;\Z) \to H^2(Y;\Z)$ has order $m = \sqrt p$.

That is, the set of \spinc structures on $Y$ that extend to $Z$ has order $m$, and since $H^2(Y;\Z)$ is cyclic, so is the image of $H^2(Z;\Z)\to H^2(Y;\Z)$.
\end{proof}

%% file: d.tex
\section{Correction terms in Heegaard Floer homology}\label{s:prelim}

Heegaard Floer homology is a family of invariants of 3--manifolds introduced by Ozsv\'ath and Szab\'o \cite{OzsvathSzabo-HF}; in this paper we are concerned with the `plus' version, which associates to a rational homology sphere $Y$ equipped with a \spinc structure $\ft$ a $\Q$--graded $\Z[U]$--module $\HFp\Yt$. Recall that the action of $U$ decreases the degree by 2.

The group $\HFp\Yt$ further splits as a direct sum of $\Z[U]$--modules $\Tp\oplus\HFpred\Yt$, where $\Tp = \Z[U,U^{-1}]/U\cdot\Z[U]$. The degree of the element $1\in \Tp$ is called the \emph{correction term} of $\Yt$, and it is denoted by $d\Yt$.

\begin{thm}[\cite{OzsvathSzabo-absolutely}]
The correction term satisfies the following properties:
\begin{itemize}\itemsep -0,5pt
\item $d(Y,\overline\ft) = d\Yt = -d(-Y,\ft)$, where $\overline\ft$ is the conjugate of $\ft$;
\item if $(W,\fs)$ is a negative definite \spinc $4$--manifold with boundary $\Yt$, then
\[
c_1(\fs)^2 + b_2(W) \le 4d\Yt.
\]
\end{itemize}
In particular $d\Yt$ is invariant under \spinc rational homology cobordisms.
\end{thm}

\begin{cor}\label{c:QHDvsd}
If $W$ is a rational homology ball with boundary $Y$, and $\ft$ is a \spinc structure on $Y$ that extends to $W$, then $d\Yt=0$.
\end{cor}

When $Y$ is obtained as rational surgery along a knot $K$ in $S^3$, one can recover the correction terms of $Y$ in terms of a family of invariants of $K$, first introduced by Rasmussen \cite{Rasmussen-Goda} and then further studied by Ni and Wu \cite{NiWu}, and Hom and Wu \cite{HomWu}. We call these invariants $\{V_i(K)\}_{i\ge 0}$, adopting Ni and Wu's notation instead of Rasmussen's --- who used $h_i(K)$ instead --- as this seems to have become more standard.

We denote with $\unknot$ the unknot.

\begin{thm}[\cite{Rasmussen-Goda, NiWu}]
The sequence $\{V_i(K)\}_{i\ge 0}$ takes values in the non-negative integers and is eventually $0$. Moreover, $V_{i}(K)-1 \le V_{i+1}(K) \le V_i(K)$ for every $i\ge 0$.

For every rational number $p/q$ and for an appropriate indexing of \spinc structures on $S^3_{p/q}(\unknot)$ and $S^3_{p/q}(K)$, we have
\begin{equation}\label{e:NiWu}
d(S^3_{p/q}(K),i) = -2\max\{V_{\lfloor i/q\rfloor}(K),V_{\lceil (p-i)/q\rceil}(K)\} + d(S^3_{p/q}(\unknot),i).
\end{equation}
\end{thm}

\begin{defn}[\cite{HomWu}]
The minimal index $i$ such that $V_i(K)=0$ is called $\nu^+(K)$.
\end{defn}

In order to know the correction terms of $S^3_{p/q}(K)$ it suffices to know the values of $V_i(K)$ for each $i$ as well as the values of the correction terms of lens spaces.

\begin{prop}[{\cite[Proposition~4.8]{OzsvathSzabo-absolutely}}]
Let $p,q,i$ be integers with $p>q>0$, $\gcd(p,q) = 1$ and $0\le i<p+q$. Denote with $r$ and $j$ the reductions of $p$ and $i$ mod $q$ respectively. Then, for an appropriate indexing of \spinc structures,
\begin{equation}\label{e:Lrecursion}
d(L(p,q),i) = \frac14 - \frac{(p+q-2i-1)^2}{4pq} - d(L(q,r),j).
\end{equation}
\end{prop}

\begin{rem}
Notice that in \cite{OzsvathSzabo-absolutely, NiWu, HomWu} the lens space $L(p,q)$ is defined as obtained by doing $p/q$--surgery along the unknot $\unknot$. We adopt the convention that $L(p,q)$ is obtained by doing $-p/q$--surgery along $\unknot$, following \cite{Neumann, Lisca-ribbon} (among others): this explains the sign difference in \cref{e:Lrecursion} with respect to the original.
\end{rem}

%% file: rational.tex
\section{Generalities on rational surgeries}\label{s:rational}

In this section we study correction terms of lens spaces; we then turn to surgeries along arbitrary knots in the 3--sphere.

\subsection{Correction terms of lens spaces}\label{ss:dlens}
Recall from~\cite{Lisca-ribbon} that the function $I: \Q_{>1}\to\Z$ is defined in terms of the negative continued fraction expansion as follows.
If $p/q \in \Q$ is larger than 1, we denote with $[a_1,\dots,a_n]^-$, where $a_i\ge 2$ for each $i$, its negative continued fraction expansion, that is $p/q = a_1 - \frac1{a_2-\dots}$.
If $p/q = [a_1,\dots,a_n]^-$, we let
\[
I(p/q) = \sum_{i=1}^n (a_i-3).
\]
Given a rational number $p/q>1$, we associate to $p/q$ the linear plumbing $P(p/q)$ represented by the diagram with $n$ vertices and weights $-a_1,\dots,-a_n$. Its plumbing graph is the following:
\[
\xygraph{
!{<0cm,0cm>;<1cm,0cm>:<0cm,1cm>::}
!~-{@{-}@[|(2.5)]}
!{(1.5,0) }*+{\bullet}="a1"
!{(3,0) }*+{\dots}="a2"
!{(4.5,0) }*+{\bullet}="c1"
!{(1.5,0.4) }*+{-a_1}
!{(4.5,0.4) }*+{-a_n}
"a2"-"a1"
"a2"-"c1"
}
\]
Recall that $P(p/q)$ is the negative canonical plumbing with boundary $L(p,q)$ and that its associated intersection form is negative definite.

\begin{defn}
We say that a rational number $p/q = [a_1,\dots,a_n]^->1$ is \emph{embeddable} if there exists an embedding of the integral lattice associated to $P(p/q)$ in $(\Z^n,-I)$.
\end{defn}

\begin{rem}\label{r:Donaldson2}
We observe here that if $S^3_{p/q}(K)$ bounds a rational homology ball, then $p/q$ is embeddable.
To this end, recall that one can express a rational surgery along a knot $K$ as an integral surgery along a link as follows: consider the negative continued fraction expansion $p/q = [a_1,\dots,a_n]^-$; then $S^3_{p/q}(K)$ is obtained by doing integral surgery on the link shown in \cref{f:rational-to-integral}.
In particular, the intersection form of the corresponding 4--dimensional 2--handlebody is the same as the intersection form of $P(p/q)$, and the claim follows from the proof of \cref{p:Donaldson}.
\end{rem}

\begin{figure}[h]
\labellist
\pinlabel $K$ at 28 49
\pinlabel $\dots$ at 236 51
\pinlabel $a_1$ at 103 100
\pinlabel $a_2$ at 130 93
\pinlabel $a_n$ at 340 93
\endlabellist
\centering
\includegraphics[scale=0.9]{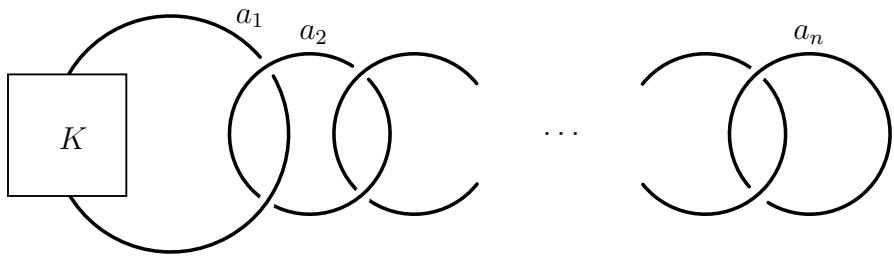}
\caption{The integral surgery picture for $p/q$--surgery along $K$, where $p/q=[a_1,\dots,a_n]^-$.}\label{f:rational-to-integral}
\end{figure}

\begin{lemma}\label{l:Lisca}
If $p/q>1$ is embeddable and $I(p/q)<0$, then $L(p,q)$ bounds a rational homology ball.
\end{lemma}

This result is implicit in Lisca's paper \cite{Lisca-ribbon}; here we adopt his notation.

\begin{proof}[Proof (sketch)]
\cite[Theorem 6.4]{Lisca-ribbon} shows that the \emph{standard subset} associated to $p/q$ with $I(p/q)<0$ can be \emph{contracted} to the standard subset associated to $4/3$; in \cite[Section 7]{Lisca-ribbon} this fact is used to characterise the set of possible continued fraction expansions of $p/q$; in \cite[Section 8]{Lisca-ribbon}, by direct inspection it is shown that each of these lens spaces does in fact bound a rational homology ball.
\end{proof}

\begin{lemma}\label{l:qgem}
Let $p = m^2$ and $q$ be positive integers. If there exists $0\le k < m$ such that
\[
d(L(p,q),k) = d(L(p,q),{k+m}) =  \dots = d(L(p,q),{k+(m-1)m}) = 0,
\]
then $q\ge m-1$. In particular, this holds if $L(p,q)$ bounds a rational homology ball.
\end{lemma}

\begin{proof}
By \cref{p:extension} and \cref{c:QHDvsd}, if $L(p,q)$ bounds there exists an integer $0\le k_0 < p$ such that $d(L(p,q),k_0+hm) = 0$ for each integer $h$ (where we think of the \spinc structures as being cyclically labelled). 
In particular, we can assume that $k_0=k$ is the remainder of the division of $k_0$ by $m$, and hence $0\le k < m$. 
This shows that the second part of the statement follows from the first.

Suppose now that the assumption hold and that, by contradiction, $q\le m-2$. In this case, $k+qm < p$, and hence
\[ d(L(p,q),k) = d(L(p,q),k+qm) = 0. \]
We now apply \cref{e:Lrecursion} to both correction terms above and subtract them, to obtain
\[
\frac{(k+qm)^2}{pq} - \frac{k^2}{pq} - mq\left(\frac1{pq}-\frac1p-\frac1q\right) = 0,
\]
from which one gets
\[
0 = 2k+mq+1-m^2 - q < 0,
\]
yielding a contradiction.
\end{proof}

\begin{lemma}\label{l:dLpq-bound}
Whenever $q,r$ are coprime and $0<r<q$, for every $0\le j\le q-1$ we have $4|d(L(q,r),j)| \le q-1$. Moreover, if equality is attained then $r$ is either $1$ or $q-1$.
\end{lemma}

\begin{proof}
The statement is symmetric with respect to the involution $r\mapsto q-r$, hence we can assume $r<q/2$. Moreover, if $r=1$ and if $q=5, r=2$ the result can be readily checked from \cref{e:Lrecursion}; hence, we can suppose $q\ge 6$ and $2\le r < q/2$.

We will proceed by induction on $q$. We now apply the recursion for $0\le j<q$:
\[
d(L(q,r),j) = \frac14 - \frac{(q+r-2j-1)^2}{4qr} - d(L(r,s),k)
\]

Notice that the second summand is bounded from above by 0 and from below by $-\frac{(q+r-1)^2}{4qr}$, while the last summand is bounded by $\frac{r-1}4$ in absolute value, by the inductive assumption.

Summing up,
\[
4d(L(q,r),j) \le 1+0+r-1 = r < q-1
\]
and
\begin{align*}
-4d(L(q,r),j) \le& -1 + \frac{(q+r-1)^2}{qr} + r-1 =\\
 = & -1 + \left(\frac {q-2}r + r\right) + 2 + \frac rq - \frac2q+\frac1{qr}-1.
\end{align*}
Now notice that $\frac{q-2}r + r$ is bounded from above by $\frac{2(q-2)}{q-1} + \frac{q-1}2 < \frac{q+3}2$, that $\frac rq < \frac12$, and that $\frac1{qr} < \frac2r$. Therefore,
\[
-4d(L(q,r),j) < \frac q2 + \frac32 + \frac12 = \frac q2 + 2, 
\]
and the latter quantity is bounded by $q-1$ if $q\ge 6$.
%
%
\end{proof}

\begin{lemma}\label{l:dLpq-zero}
Let $p/q$ be an embeddable rational number. Then at least one of the correction terms of $L(p,q)$ vanishes.
\end{lemma}

\begin{proof}
Let $[a_1,\dots,a_n]^-$ be the negative continued fraction expansion of $p/q$.

The plumbing $P = P(p/q)$ has no bad vertices in the sense of~\cite{OzsvathSzabo-plumbed}, hence the correction terms of its boundary can be computed in terms of the squares of the characteristic vectors in $H_2(P)$. More specifically, it follows from~\cite[Corollary 1.5]{OzsvathSzabo-plumbed} that
\begin{equation}\label{e:plumbingd}
d(-Y,\ft) = \frac14(\max c_1^2(\fs)+n)
\end{equation}
where the maximum is taken over all \spinc structures on $P$ whose restriction to $Y$ is $\ft$.

Fix an embedding of the intersection lattice $L\cong H_2(P;\Z)$ of $P$ in the negative definite diagonal lattice $D = (\Z^n,-I)$, with a fixed orthonormal basis $\{e_1,\dots,e_n\}$. The vector space $L_\Q := L\otimes\Q = D\otimes\Q$ comes with a natural scalar product, and there is an embedding $\iota$ of the dual lattice $L^*$ into $L_\Q$ that turns evaluations into products. That is, for every $\lambda\in L^*$ and $l\in L$ we have
\[
\lambda(l) = \iota(\lambda)\cdot l
\]
The sum $s = e_1+\dots+e_n\in L_\Q$ is in $L^*$, and it is a characteristic covector (since it is a characteristic vector in $D$); $s$ also maximises the norm among characteristic vectors in $D$, and in particular it maximises the norm in the set $s + L$.
By \cref{e:plumbingd} the associated \spinc structure on $-Y$ has correction term $(s^2+n)/4 = 0$.
\end{proof}

\begin{lemma}\label{l:congruence}
If $p=m^2$ is a square and $d(L(p,q),i)$ is an integer, then $2i+1\equiv q\pmod m$.
\end{lemma}

\begin{proof}
Recall that correction terms of lens spaces can be computed in terms of Dedekind and Dedekind--Rademacher sums; we only need two properties of these sums, and we refer the reader to \cite{Jabuka-when, Jabuka-HFcorr} for precise definitions and proofs.

More specifically, it is shown in \cite[Theorem 1.2]{Jabuka-HFcorr} that
\begin{equation}\label{e:DRsum}
d(L(q,t),j) = 2s(t,q;j) + s(t,q) -\frac1{2q}
\end{equation}
where $s(t,q;j), s(t,q)$ are two rational numbers such that $6qs(t,q)$ and $12qs(t,q;j)$ are both integers (see \cite[Section 2]{Jabuka-when}, where the notation for $s(t,q;j)$ is $r_j(t,q)$). In particular, $12qd(L(q,t),j)$ is an integer for each $j$.

We apply \cref{e:DRsum} setting $t = p$ and $j$ equal to the reduction of $i$ modulo $q$: if $d(L(p,q),i)$ is an integer, then so is $12qd(L(p,q),i)$. However, applying \cref{e:Lrecursion},
\[
12qd(L(p,q),i) = 3q - 3\frac{(p+q-2i-1)^2}p - 12qd(L(q,t),j),
\]
and this implies that $3\left(\frac{p+q-2i-1}m\right)^2$ is an integer, hence that $m$ divides $m^2+q-2i-1$, which is equivalent to the thesis.
\end{proof}

\begin{lemma}\label{l:integrality}
If $p=m^2$ and $d(L(p,q),i)\in\Z$ then $d(L(p,q),i+m)\in\Z$. Moreover, if $d(L(p,q),i)$ and $d(L(p,q),i')$ are both integers, then $m\mid(i-i')$.
\end{lemma}

The proof of this lemma follows closely the proof of \cite[Corollary 1.8]{Jabuka-HFcorr}.

\begin{proof}
Recall that \cite[Lemma 2.2]{Jabuka-HFcorr} asserts that $2pd(L(p,q),i)$ and $2qd(L(q,r),j)$ are integers for every $i$ and $j$.
We refine this by applying \cref{e:Lrecursion} with $i<p$ and $i+q$: subtracting the two equations and multiplying by $p$ we obtain that
\[
pd(L(p,q),i) \equiv pd(L(p,q),i+q) \pmod 1,
\]
i.e. either $pd(L(p,q),i)$ is integer for each $i$, or it is a half-integer for each $i$.

Consider \cref{e:Lrecursion} for $i$ and $i'$, and let $j$ and $j'$ be the reductions of $i$ and $i'$ modulo $q$; multiplying by $q$ and subtracting the two identities we get:
\begin{equation}\label{e:ii'}
q(d(L(p,q),i)-d(L(p,q),i')) = {\textstyle\frac{(i'-i)(p+q-i-i'-1)}p} - qd(L(q,r),j)+qd(L(q,r),j').
\end{equation}

We are now in position to prove that if $d(L(p,q),i)$ is an integer, then also $d(L(p,q),i+m)$ is.
Set $i' = i+m$ in \cref{e:ii'}.
Notice that $d(L(p,q),i)$ is an integer by assumption, and the last two summands on the right-hand side are either 
both integers or both half-integers, and in either case their sum is an integer. Moreover, $i-i'=m$ by assumption, 
and $m$ divides $p+q-i-i'-1 = p+2-2i-m-1$ by \cref{l:congruence}. It follows that $qd(L(p,q),i')$ is an integer, 
and since $2pd(L(q,p),i')$ is an integer then so is $d(L(p,q),i')$.

The second part of the statement is obvious when $p$ is odd.
Suppose $p$ is even and $d(L(p,q),i)$ is an integer; combining \cref{l:congruence} with the first part of the statement, it is enough to show that $d(L(p,q),i+\frac m2)$ is not an integer.

Let $m=2n$ and plug in $i'=i+n$ in \cref{e:ii'}: multiplying by 2, the left-hand side is congruent to $-2qd(L(p,q),i')$ modulo 1, and the right-hand side is congruent to $2\frac{n(p+q-2i-n-1)}{4n^2} = \frac{p+q-2i-1}{2n} - \frac1{2}\equiv \frac12$ modulo 1.
The first summand is integer by \cref{l:congruence}, hence $2q(L(p,q),i')$ cannot be an integer.
\end{proof}

\subsection{Correction terms of rational surgeries}

\begin{lemma}\label{l:dlens-to-surgery}
Let $p=m^2$. The correction term $d(S^3_{p/q}(K),i)$ is an integer if and only if $d(S^3_{p/q}(\unknot),i)$ is. In particular, if $S^3_{p/q}(K)$ has an integral correction term, there are exactly $m$ integers $i_0, i_0+m,\dots,i_0+(m-1)m$ such that $d(S^3_{p/q}(K),i)\in\Z$.
\end{lemma}

\begin{proof}
Using Equation~\eqref{e:NiWu}, we observe that $d(S^3_{p/q}(K),i)$ is an integer if and only if $d(S^3_{p,q}(\unknot),i)$ is, and applying \cref{l:integrality} we conclude the proof.
\end{proof}

\begin{prop}\label{p:Ipq<0}
Let $0<q<p$ be coprime integers with $I(p/q)<0$, and suppose that $S^3_{p/q}(K)$ bounds a rational homology ball; then $\nu_+(K) = 0$.
\end{prop}

\begin{proof}
By \cref{r:Donaldson2}, $p/q$ is embeddable; since $I(p/q)<0$, \cref{l:Lisca} shows that $L(p,q) = S^3_{-p/q}(\unknot)$ bounds a rational homology ball. It follows from \cref{l:qgem} that $q\ge m-1$, where $m^2 = p$. Notice also that $I(\frac{m^2}{m-1})=1>0$, hence the case $p/q = m^2/(m-1)$ is excluded, and we can suppose $q>m$.

By \cref{p:extension}, $m$ equally spaced correction terms of $L(p,q)$ vanish, and there is an $0\le i\le m-1 < q$ such that
\[
0 = d(S^3_{p/q}(K),i) = d(L(p,p-q),i) - 2V_{\lfloor i/q \rfloor} (K) = 0-2V_0(K),
\]
hence $V_0(K)=0$, or, equivalently, $\nu_+(K)=0$.
\end{proof}

\begin{rem}
We observe that the following partial converse to the proposition above holds: if $\nu^+(K) = 0$ and $S^3_{p/q}(K)$ bounds, then also $L(p,q)$ bounds.
In fact, since $\nu^+(K) = 0$, the correction terms of $S^3_{p/q}(K)$ are the same as the correction terms of $-L(p,q) = S^3_{p/q}(\unknot)$.
Since $S^3_{p/q}(K)$ bounds, it has $\sqrt p$ vanishing correction terms, hence so does $\pm L(p,q)$ and Greene has recently shown that this in turn implies that $L(p,q)$ bounds a rational homology ball~\cite{Greene}.
\end{rem}

We are now ready to prove a quantitative version of \cref{t:pqfinite}.

\begin{thm}\label{t:pqbounds}
Let $K$ be a knot in $S^3$ with $\nu^+(K) = \nu$, and $p,q$ be positive, coprime integers with $p=m^2$; suppose that $S^3_{p/q}(K)$ bounds a rational homology ball. Then, if $\nu>0$:
\[
\sqrt{(2\nu-1)q} < m < \frac{q+2 + \sqrt{q^2 + 8q\nu+8}}2,
\]
while if $\nu=0$ then $m\le q+1$.
\end{thm}

\begin{proof}
We treat the case $\nu = 0$ first. If $\nu = 0$, then the correction terms of $S^3_{p/q}(K)$ are the same as the correction terms of $-L(p,q)$, and \cref{p:extension} implies that there exists $0\le k<m$ such that
\[
d(L(p,q),\ft_k) = d(L(p,q),\ft_{k+m}) =  \dots = d(L(p,q),\ft_{k+(m-1)m}) = 0.
\]
It follows from \cref{l:qgem} that $q\ge m-1$.

We now treat the case $\nu>0$.
We first rule out the possibility that $p=1$; in fact, if $p=1$, then $S^3_{p/q}(K)$ is an integral homology sphere, and by~\eqref{e:NiWu} its unique correction term is $-2V_0(K) \neq 0$.
In particular, it never bounds a rational homology ball.

We can now suppose $p>1$.
We begin by proving the left-hand side inequality: if $p/q>1$, by \cref{r:Donaldson2} we know that $p/q$ is embeddable.

If $p<q$, we can write\footnote{In \cref{ss:dlens} we defined negative continued fraction expansions only for $p/q>1$ and we required that all terms in the expansion be larger than 1.
We make an exception here to treat the case $p/q<1$.}
$p/q = [1,2,\dots,2,a_1+1,\dots,a_\ell]^-$ for some $a_i\ge 2$.
The plumbing graph associated to this continued fraction expansion is negative definite and the intersection form of the plumbing embeds, according to \cref{r:Donaldson2}. 
By successively blowing down we obtain an embedding of $P(p/q')$, where $q' = [a_1,\dots,a_\ell]^-$, and 
$0 < q' = q-\lfloor q/p \rfloor p <  p$. In particular, $p/q'$ is embeddable and $L(p,q') = L(p,q)$; by extension, 
we say that also $p/q$ is embeddable.

Regardless of whether $p<q$ or $p>q$, it follows from \cref{l:dLpq-zero} that at least one of the correction terms of $S^3_{p/q}(\unknot)$ vanishes; choose $i_0$ such that $d(S^3_{p/q}(\unknot),i_0) = 0$.
By \cref{l:dlens-to-surgery}, the values of $i$ for which $d(S^3_{p/q}(\unknot),i)$ is an integer are all integers of the form $i = i_0+km$ between $0$ and $m-1$.

Since $S^3_{p/q}(K)$ bounds a rational homology ball, by \cref{p:extension} there are $m$ \spinc structures on $S^3_{p/q}(K)$ with vanishing correction terms. In particular, these have to be the $m$ \spinc structures with integral correction term, i.e. the ones labelled with $i_0+km$.

We now apply \cref{e:NiWu} with $i=i_0$:
\begin{align*}
0 = d(S^3_{p/q}(K),i_0) &= -2\max\{V_{\lfloor i_0/q\rfloor}(K), V_{\lceil (p-i_0)/q \rceil}(K)\} + d(S^3_{p/q}(\unknot),i_0) =\\
 &= -2\max\{V_{\lfloor i_0/q\rfloor}(K), V_{\lceil (p-i_0)/q \rceil}(K)\},
\end{align*}
hence $V_{\lfloor i_0/q\rfloor}(K) = V_{\lceil (p-i_0)/q \rceil}(K) = 0$. These equalities imply that $i_0 \ge q\nu$ 
and $p-i_0 > (\nu-1)q$ respectively, and summing them yields $p > (2\nu-1)q$.

Notice that this implies, a posteriori, $p>q$.

We now turn to the right-hand side inequality, keeping in mind that we can assume $p>q$. From \cref{e:Lrecursion}
\[
d(L(m^2,q),i) = \frac14 - \frac{(m^2+q-2i-1)^2}{4m^2q} - d(L(q,r),j).
\]
Hence, applying \cref{l:dLpq-bound}, $d(L(m^2,q),i)$ can vanish only if $(m^2+q-2i-1)^2 \le |m^2q(4d(L(q,r),j)+1)| \le m^2q^2$, that is only if
\[
m^2-mq \le 2i+1-q \le m^2+mq.
\]

Moreover, by \cref{l:congruence}, if $d(L(m^2,q),i)$ is integral then $2i+1-q\equiv 0\pmod m$, hence if $2i$ is one of $m^2-mq+q-1-m$ or $m^2-mq+q-1-2m$, the corresponding correction term of $L(m^2,q)$ is integral but does not vanish.

If $S^3_{p/q}(K)$ bounds a rational homology ball, then for each $i<m^2-mq-q-1$ for which the corresponding \spinc structure extends, Equation~\eqref{e:NiWu} shows that the invariant $V_{\lfloor i/q\rfloor}(K)$ is nonzero, and in particular $\lfloor \frac iq\rfloor < \nu$. Thus, the following inequality holds
\[
\frac{m^2-mq+q-1-2m}2 = \min\left\{\frac{m^2-mq+q-1-2m}2,\frac{m^2-mq+q-1-m}2\right\} < q\nu
\]
from which the right-hand side inequality follows.
\end{proof}

\begin{rem}
We can gain something more from the right-hand side inequality if we assume that $m^2\not\equiv \pm1\pmod q$. In this case, we know that the inequality coming from \cref{l:dLpq-bound} is strict, hence we can use $m^2-mq+q-1$ instead of $m^2-mq+q-1-2m$, hence obtaining
\[
m < \frac{q+1 + \sqrt{(q-1)^2+8q\nu+4}}2
\]
However, this assumption is not particularly harmful, since whenever $m^2\equiv \pm1\pmod q$
 the correction terms of $L(m^2,q)$ can be computed explicitly by applying the symmetry formula once, together with the computation of $d(L(q,1),j)$.
\end{rem}

\begin{cor}
If $p\equiv1\pmod q$, then $S^3_{p/q}(T_{3,2})$ bounds a rational homology ball if and only if $p=(q+1)^2$ or $p/q = 25/3$.
\end{cor}

\begin{proof}
First, we observe that the case $p=1$ is excluded by \cref{t:pqbounds}.

We are now going to prove the ``only if'' direction. Let $p=m^2$, $K = T_{3,2}$, and suppose that $S^3_{p/q}(K)$ bounds a rational homology ball.
Since $m^2/q = [\lceil m^2/q \rceil, 2^{[q-1]}]^-$, we have that $I(m^2/q) = \lceil m^2/q\rceil - q-2$, and hence $I(p/q)<0$ if $m\le q-1$.
Since $\nu^+(K) = 1$, \cref{p:Ipq<0} shows that $m\ge q+1$.

On the other hand, the right-hand side inequality in \cref{t:pqbounds} says that
\[
m < \frac{q+2 + \sqrt{q^2+8q+8}}2 < \frac{q+2 + \sqrt{q^2+8q+16}}2 = q+3,
\]
hence the only possible values for $m$ are $q+1$ and $q+2$. Observe that if $m=q+2$, then the condition $m^2\equiv 1\pmod q$ reads $(q+2)^2 \equiv 1 \pmod q$, from which $q\mid3$.
Therefore, the ``only if'' direction is proved.

When $q\ge 5$, the 3--manifold $S^3_{(q+1)^2/q}(K)$ bounds the negative definite plumbing associated to the graph (see \cref{ss:seifert} below for the details)
\[
 \xygraph{
!{<0cm,0cm>;<1cm,0cm>:<0cm,1cm>::}
!~-{@{-}@[|(2.5)]}
!{(0,1.5) }*+{\bullet}="c"
!{(1.5,1.5) }*+{\bullet}="d"
!{(3,1.5) }*+{\bullet}="e"
!{(4.5,1.5) }*+{\dots}="f"
!{(6,1.5) }*+{\bullet}="x"
!{(7.5,1.5) }*+{\bullet}="g"
!{(1.5,0) }*+{\bullet}="h"
!{(0,1.9) }*+{-2}
!{(1.5,-0.4) }*+{-3}
!{(1.5,1.9) }*+{-2}
!{(3,1.9) }*+{-2}
!{(6,1.9) }*+{-2}
!{(7.5,1.9) }*+{-(q+1)}
!{(4.5,2.5) }*+{\overbrace{\hphantom{--------}}^{q-5}}
"c"-"d"
"d"-"e"
"e"-"f"
"f"-"x"
"d"-"h"
"x"-"g"
}\]
which Park, Shin and Stipsicz have proven to bound a rational homology ball (see \cite{ParkShinStipsicz}, Figure 1, third graph from the bottom). 
When $q=4$, $S^3_{25/4}(K) = L(25,14)$, and this bounds by the work of Lisca \cite{Lisca-ribbon}.

When $q=3$, $-S^3_{16/3}(K)$ bounds the plumbing
\[
 \xygraph{
!{<0cm,0cm>;<1cm,0cm>:<0cm,1cm>::}
!~-{@{-}@[|(2.5)]}
!{(0,1.5) }*+{\bullet}="c"
!{(1.5,1.5) }*+{\bullet}="d"
!{(3,1.5) }*+{\bullet}="e"
!{(4.5,1.5) }*+{\bullet}="f"
!{(1.5,0) }*+{\bullet}="h"
!{(0,1.9) }*+{-2}
!{(1.5,-0.4) }*+{-2}
!{(1.5,1.9) }*+{-3}
!{(3,1.9) }*+{-2}
!{(4.5,1.9) }*+{-2}
"c"-"d"
"d"-"e"
"e"-"f"
"d"-"h"
}\]
and also the plumbing
\[
 \xygraph{
!{<0cm,0cm>;<1cm,0cm>:<0cm,1cm>::}
!~-{@{-}@[|(2.5)]}
!{(0,1.5) }*+{\bullet}="c"
!{(1.5,1.5) }*+{\bullet}="d"
!{(3,1.5) }*+{\bullet}="e"
!{(4.5,1.5) }*+{\bullet}="f"
!{(6,1.5) }*+{\bullet}="x"
!{(7.5,1.5) }*+{\bullet}="g"
!{(1.5,0) }*+{\bullet}="h"
!{(0,1.9) }*+{-2}
!{(1.5,-0.4) }*+{-2}
!{(1.5,1.9) }*+{-1}
!{(4.5,1.9) }*+{-2}
!{(3,1.9) }*+{0}
!{(6,1.9) }*+{-2}
!{(7.5,1.9) }*+{-2}
"c"-"d"
"d"-"e"
"e"-"f"
"f"-"x"
"d"-"h"
"x"-"g"
}
\]

It follows from \cite[Proposition 4.6]{Paolo} that $-S^3_{16/3}(K)$ is rationally homology cobordant to $L(4,3)$, hence it bounds a rational homology ball.

When $q=2$, $S^3_{9/2}(K)$ is known to be diffeomorphic to $-S^3_9(K)$ \cite{Mathieu}, 
hence it bounds a rational homology ball as well (e.g. see \cref{t:kq+1}).
Finally, $S^3_{25/3}(K)$ bounds a rational homology ball, too; this was included in Casson--Harer's list \cite{CassonHarer}, setting $p=3,s=4,k=5$ in their first family.
\end{proof}


%% file: integral.tex
\section{Integral surgeries}\label{s:integral}

In this section we study integral surgeries along knots in the 3--sphere, and we then focus on alternating knots and torus knots.

The following is a quantitative version of \cref{p:mvsnu1}.

\begin{thm}\label{p:mvsnu}
Let $m$ be a positive integer, $K\subset S^3$ be a knot, and suppose that $S^3_{m^2}(K)$ bounds a rational homology ball. Let $\nu = \nu_+(K)$. Then
\[
1+\sqrt{1+8\nu} \le 2m < 3+\sqrt{9+8\nu}
\quad
{\textrm or, equivalently, }
\quad
0 \le \frac{m(m-1)}2 - \nu < m.
\]
\end{thm}

\begin{lemma}\label{l:integral2}
The correction term $d(S^3_{m^2}(K),i)$ is integral if and only if $i = m(m-2k-1)/2$ for some integer $k$.
\end{lemma}

\begin{proof}
By a direct computation from \cref{e:Lrecursion},
\begin{equation}\label{e:dLm-1}
d(S^3_{m^2}(\unknot),i) = \frac{(m^2-2i)^2}{4m^2} - \frac14.
\end{equation}
This is an integer if and only if $m^2-2i \equiv m \pmod{2m}$, which in turn is equivalent to the condition $i = m(m-2k-1)/2$ for some integer $k$; the same holds for $S^3_{m^2}(K)$ in light of \cref{l:dlens-to-surgery}.
\end{proof}

Applying~\eqref{e:NiWu} and~\eqref{e:dLm-1}, we obtain:
\begin{equation}\label{e:NiWuZ}
d(S^3_{m^2}(K),i) = -2V_i(K) + \frac{(m^2-2i)^2}{4m^2} - \frac14.
\end{equation}

\begin{lemma}\label{l:vanishingds}
If $S^3_{m^2}(K)$ bounds a rational homology ball, then
\[
d\left(S^3_{m^2}(K), \frac{m(m-2k-1)}2\right) = 0
\]
for each $k=0,\dots, \lfloor \frac{m-1}2 \rfloor$.
\end{lemma}

\begin{proof}
It follows from \cref{p:extension} that if $S^3_{m^2}(K)$ bounds a rational homology ball $Z$, then $m$ of its \spinc structures extend to a $Z$;
\cref{c:QHDvsd}, in turn, implies that the corresponding correction terms vanish.
\cref{l:integral2} now pins down exactly which correction terms of $S^3_{m^2}(K)$ can be integral: since there are exactly $m$ of them, they must all vanish.
\end{proof}

\begin{proof}[Proof of \cref{p:mvsnu}]
Notice that the first chain of inequalities is obtained from the second by solving for $m$; hence, we set out to prove the latter.

It follows from \cref{l:vanishingds} the correction term of $S^3_{m^2}(K)$ corresponding to $\frac{m(m-1)}2$ vanishes; if $m\ge 3$, the one corresponding to $\frac{m(m-3)}2$ vanishes, too.

Since the correction term corresponding to $\frac{m(m-1)}2$ vanishes, from~\eqref{e:NiWuZ} we obtain
\[
0 = d\left(S^3_{m^2}(K),\textstyle \frac{m(m-1)}2\right) = -2V_{\frac{m(m-1)}2}(K) + d\left(S^3_{m^2}(\unknot),\textstyle \frac{m(m-1)}2\right) = -2V_{\frac{m(m-1)}2}(K),
\]
from which we obtain the inequality $\frac{m(m-1)}2 \ge \nu$.

Notice now that if $m<3$ then the second inequality is automatically satisfied by $m$, so to prove the remaining part of the statement we can assume $m\ge3$. From the vanishing of the correction term corresponding to $\frac{m(m-3)}2$ from~\eqref{e:NiWuZ} we obtain
\begin{align*}
0 &= d\left(S^3_{m^2}(K),\textstyle \frac{m(m-3)}2\right) = -2V_{\frac{m(m-3)}2}(K) + d\left(S^3_{m^2}(\unknot),\textstyle \frac{m(m-3)}2\right) =
 2-2V_{\frac{m(m-3)}2}(K),
\end{align*}
from which $\frac{m(m-1)}2 - m = \frac{m(m-3)}2 < \nu$.
\end{proof}

We are now ready to prove \cref{p:mvsnu1}.
\begin{proof}[Proof of \cref{p:mvsnu1}]
Let $\nu = \nu^+(K)$. \cref{p:mvsnu} asserts that, if $S^3_{m^2}(K)$ bounds a rational homology ball, then $2m$ is in the interval $I_\nu = [1+\sqrt{1+8\nu}, 3+\sqrt{9+8\nu})$; $I_\nu$ has length
\[
2+\sqrt{9+8\nu}-\sqrt{1+8\nu} = 2 + \frac8{\sqrt{9+8\nu}+\sqrt{1+8\nu}} \le 4,
\]
hence it contains at most two even integers, and, if it contains two, they are consecutive.

Moreover, if $I_\nu$ contains two even integers, then we have
\[
1+\sqrt{1+8\nu} \le 2m < 2m+2 < 3+\sqrt{9+8\nu}.
\]
The left-hand side inequality can be rearranged as $1+8\nu \le (2m-1)^2$, while the right-hand side inequality reads $(2m-1)^2 < 9+8\nu$; that is,
\[
1+8\nu \le (2m-1)^2 < 9+8\nu.
\]
Since odd squares are congruent to 1 modulo 8, we have $(2m-1)^2 = 1+8\nu$.
\end{proof}

\begin{ex}
There exist knots that have two positive, integral surgeries bounding a rational homology ball: if $K$ is the $(q+1,q)$--torus knot, then \cref{t:kq+1} shows that both $S^3_{q^2}(K)$ and $S^3_{(q+1)^2}(K)$ bound rational homology balls.

The second example is in fact realised in an algebro-geometric fashion \cite{55letters}: the curve $C = \{x^{q+1}+y^{q}z\}$ is rational with a unique singularity at $(0,0,1)$, of type $(q+1,q)$; the boundary of an open regular neighbourhood of $C$ is $-S^3_{(q+1)^2}(K)$.
\end{ex}

\begin{rem}
When $m$ is odd and $S^3_{m^2}(K)$ bounds a rational homology ball, we can also infer that $m^2 = 1+8V_0(K)$. In particular, $1+8V_0(K)$ is a perfect square.

This follows immediately from the fact that when $m$ is odd, the \spinc structure labelled with 0 extends to the rational homology ball, hence
\[
0 = d(S^3_{m^2}(K),0) = -2V_0(K) + d(L(m^2,-1),0) = -2V_0(K) + \frac{m^2-1}4.
\]
\end{rem}

\subsection{Alternating knots}

The knot Floer homology of alternating knots is very simple, and it is fully determined by the Alexander polynomial and the signature~\cite{OzsvathSzabo-alternating}; in fact, alternating knots are \emph{Floer-thin}, in the sense that their knot Floer homology is supported on a diagonal with constant $i-j$ (we refer to \cite{OzsvathSzabo-alternating} for the notation).

More generally, \emph{quasi-alternating} knots have been shown to be Floer-thin by Manolescu and Ozsv\'ath \cite{ManolescuOzsvath}; for all quasi-alternating knots, the constant $i-j$ is equal to half of the signature.

\begin{prop}\label{p:alternating}
Let $K$ be a Floer-thin knot supported on the diagonal $i-j = \sigma/2$, and $m$ be an integer such that $S^3_{m^2}(K)$ bounds a rational homology ball. Then $m\le 5$.

Moreover:
\begin{itemize}
\item if $m=1$, $\sigma\ge 0$;
\item if $m=2$, $\sigma\ge -2$;
\item if $m=3$, either $\sigma = -2$ or $\sigma = -4$;
\item if $m=4$, either $\sigma = -6$ or $\sigma = -8$;
\item if $m=5$, $\sigma = -12$.
\end{itemize}

In particular, the statement holds when $K$ is alternating, and $\sigma = \sigma(K)$ is its signature.
\end{prop}

%

\begin{proof}
Combining results from \cite{OzsvathSzabo-alternating, ManolescuOzsvath} as in \cite[Theorem 2]{HomWu}, one has:
\begin{equation}\label{e:alternating}
V_i(K) = \max\left\{\left\lceil \frac{\tau(K) - i}2 \right\rceil,0\right\}.
\end{equation}

Suppose $m\ge5$. According to \cref{l:vanishingds}, the \spinc structures on $S^3_{m^2}(K)$ labelled by $\frac{m(m-2k-1)}2$ for $k=0,1,2$ have vanishing correction term. By~\eqref{e:NiWuZ}, one gets:
\[
V_{\frac{m(m-3)}2}(K) = 1, \quad V_{\frac{m(m-5)}2}(K) = 3
\]
But it follows from \cref{e:alternating} that $|V_i(K)-V_j(K)| \ge \lfloor |i-j|/2\rfloor$, and therefore $2 \ge \lfloor m/2\rfloor$, from which $m\le 5$.

As for the second part of the statement, recall that when $K$ is Floer-thin, one has that $2\tau(K) = -\sigma$, and moreover
\[
\nu^+(K) = \left\{
\begin{array}{ll}
\tau(K) & {\rm if}\, \tau(K) \ge 0\\
0 & {\rm if}\, \tau(K) < 0
\end{array}
\right.
\]
Bearing these in mind:
\begin{itemize}
\item if $m=1$, then $V_0(K)=0$, hence $\tau(K) \le 0$, $\sigma(K) \ge 0$;
\item if $m=2$, then $V_1(K)=0$, hence $\tau(K)\le 1$, $\sigma(K) \ge -2$;
\item if $m=3$, then $V_0(K) = 1$, from which either $\tau(K) = 1$ or $\tau(K) = 2$;
\item if $m=4$, then $V_2(K) = 1$, from which either $\tau(K) = 3$ or $\tau(K) = 4$;
\item if $m=5$, then $V_0(K) = 3$ and $V_5(K)=1$, from which $\tau(K) = 6$.\qedhere
\end{itemize}
\end{proof}

\begin{rem}
Notice that \cref{t:kq+1} shows that all cases with $\sigma < 0$ above are realised; the cases with $\sigma = 0$ are realised by the unknot. Finally, Fintushel and Stern \cite{FintushelStern} have proved that $S^3_{+1}(T_{3,-2})$ bounds a rational homology ball, hence showing that the pair $(\sigma,m) = (2,1)$ is realised, too.
\end{rem}

\subsection{Torus knots}

Now we can prove \cref{t:pvs9q}. A second proof will be given in \cref{ss:seifert}.

\begin{proof}[First proof of \cref{t:pvs9q}]
The case of integral surgeries along alternating torus knots (i.e. $T_{2k+1,2}$) has already been treated in \cref{p:alternating}, hence we will focus on the case $q>2$.

Recall that for a torus knot $T_{p,q}$ the function $V(i) = V_i(T_{p,q})$ is related to the gap counting function $I_\Gamma$ of the semigroup $\Gamma = \langle p,q\rangle \subset \Z_{\ge0}$ generated by $p$ and $q$ as follows.
Let $\nu = \nu^+(T_{p,q})$, and recall that $\nu = g(T_{p,q}) = \frac{(p-1)(q-1)}2$. The function $I_\Gamma$ is defined as
\[
I_\Gamma(j) = \#(\Z_{\ge j}\setminus \Gamma);
\]
Borodzik and Livingston \cite{BorodzikLivingston} proved that
\[
V(j) = I_\Gamma(j+\nu).
\]
In particular, $V(j) = 1$ exactly when $j+\nu$ varies between the second largest gap of the semigroup and its largest gap, i.e. between the last two elements that are not expressible as a nonnegative integer combination of $p$ and $q$. Similarly, $V(j)=2$ between the second and third largest gaps.
The largest gap is well-known to be $g_1 = pq-p-q = 2\nu-1$.
A graph of the function $V(\cdot)$ is sketched in \cref{f:graphVTpq}.

\begin{figure}[h]
\labellist
\pinlabel $V_i(T_{p,q})$ at 33 174
\pinlabel $i$ at 320 15
\pinlabel $\dots$ at 42 24
\pinlabel $\vphantom{1}_{\nu-9q}$ at 76 18
\pinlabel $\vphantom{1}_{\nu-8q}$ at 100 18
\pinlabel $\vphantom{1}_{\nu-7q}$ at 124 18
\pinlabel $\vphantom{1}_{\nu-6q}$ at 148 18
\pinlabel $\vphantom{1}_{\nu-5q}$ at 172 18
\pinlabel $\vphantom{1}_{\nu-4q}$ at 196 18
\pinlabel $\vphantom{1}_{\nu-3q}$ at 220 18
\pinlabel $\vphantom{1}_{\nu-2q}$ at 244 18
\pinlabel $\vphantom{1}_{\nu-q}$ at 268 18
\pinlabel $\vphantom{1}_{\nu}$ at 292 18
\pinlabel $\overbrace{\hphantom{--.}}^{V=1}$ at 279 45
\pinlabel $\overbrace{\hphantom{--.}}^{V=3}$ at 231 69
\pinlabel $\overbrace{\hphantom{--.}}^{V=6}$ at 159 106
\pinlabel $\overbrace{\hphantom{--.}}^{V=10}$ at 63 154
\endlabellist
\centering
\includegraphics[scale=1.25]{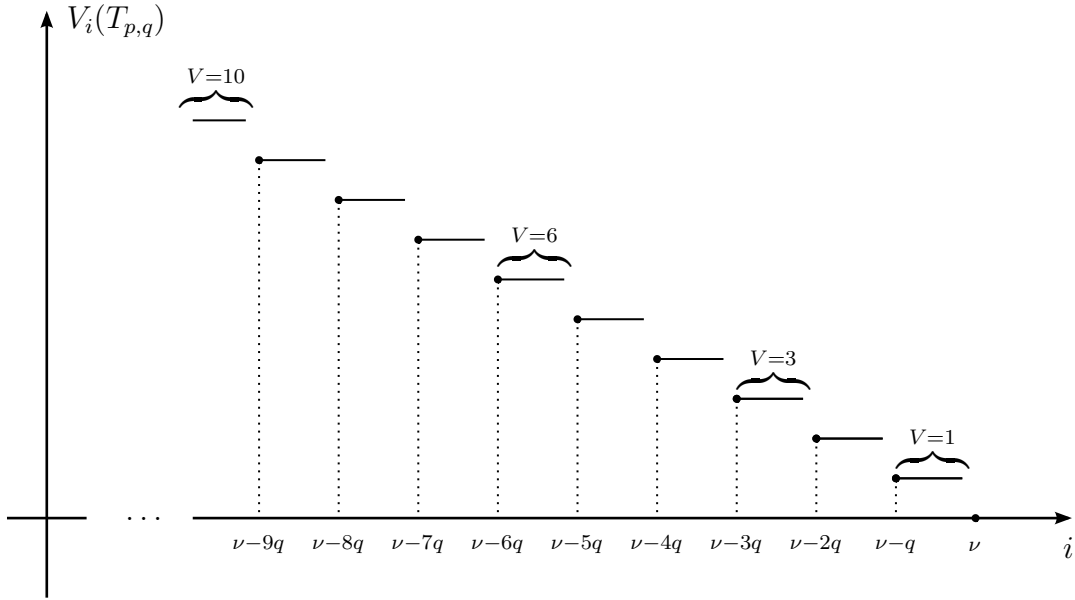}
\caption{A portion of the graph of the function $V(\cdot)$, under the assumption $p>9q$. We let $\nu = \nu(T_{p,q}) = \frac{(p-1)(q-1)}2$.}\label{f:graphVTpq}
\end{figure}

Suppose now that $p>9q$.

In what follows, we will need to use \cref{l:vanishingds} with $k=4$, and this is allowed only when $m\ge 9$.
However, we notice that if $q>3$, then $\nu \ge \frac{9q(q-1)}2 \ge 36 > 28$, and that if $q=3$ and $p>29$ then $\nu > 28$; in both cases, thanks to \cref{p:mvsnu}, $m$ is strictly larger than $\frac{1+\sqrt{1+28\cdot 8}}2 = 8$.

The two cases $q=3$, $p = 28, 29$ need separate treatment: in both instances, \cref{p:mvsnu} implies that the only possible value of $m$ is $8$, and a direct computation shows that the correction term $d(S^3_{64}(T_{p,3}),20)$ is $-4$.
That is, $S^3_{64}(T_{28,3})$ and $S^3_{64}(T_{29,3})$ do not bound rational homology balls.

Since neither $T_{28,3}$ nor $T_{29,3}$ has a surgery that bounds a rational homology ball, we can suppose that $\nu>28$, and in particular $m\ge 9$.

Recall that the semigroup of the singularity is symmetric \cite[Theorem 4.3.5]{Wall}, in the sense that $x$ belongs to the semigroup if and only if $pq-p-q-x$ does not.
Hence, since $p > 9q$, the first nine elements of the semigroup are $0,q,\dots,9q$, and therefore the nine largest gaps $g_9 < \dots < g_1$ are $g_k = (q-1)p - kq$ (see \cref{f:graphVTpq}).

This shows that if $V(i) = 10$, then $i < g_9-\nu$; if $k=1,\dots, 9$, $V(i) = k$ exactly when $g_k-\nu\le i < g_{k-1}-\nu$.

If $S^3_{m^2}(T_{p,q})$ bounds and $m\ge 9$, then \cref{l:vanishingds} implies that its correction terms in the \spinc  structures labelled by $\frac{m(m-9)}2, \frac{m(m-7)}2, \frac{m(m-5)}2, \frac{m(m-3)}2, \frac{m(m-1)}2$ must all vanish.
Applying Equation~\eqref{e:NiWuZ}, we obtain $V(\frac{m(m-9)}2) = 10$, $V(\frac{m(m-7)}2) = 6$, $V(\frac{m(m-5)}2) = 3$, $V(\frac{m(m-3)}2) = 1$.

These translate into the inequalities
\[
\begin{aligned}
{\textstyle\frac{m(m-9)}2} &< \nu-9q,\\
\nu-3q \le {\textstyle\frac{m(m-5)}2} &< \nu-2q,
\end{aligned}
\qquad
\begin{aligned}
\nu-6q &\le {\textstyle\frac{m(m-7)}2} < \nu-5k,\\
\nu-q &\le {\textstyle\frac{m(m-3)}2} < \nu.
\end{aligned}
\]

In particular, the first row of inequalities implies that $m = \frac{m(m-7)}2 - \frac{m(m-9)}2 > 3q$; the second row, on the other hand, shows that $m = \frac{m(m-3)}2 - \frac{m(m-5)}2 < 3q$, which leads to a contradiction.
%
%
%
\end{proof}

\subsection{Asymptotic classification}

Recall that \cref{p:mvsnu} asserts that if $S^3_{m^2}(K)$ bounds a rational homology ball, then
\[
\frac{1+\sqrt{1+8\nu}}2 \le m < \frac{3+\sqrt{9+8\nu}}2,
\]
from which one obtains
\[
2\nu+2-(1+\sqrt{1+8\nu}) \le (m-1)(m-2) < 2\nu+2.
\]
If we call $2\eps$ the difference $2\nu+2-(m-1)(m-2)$, we have $0<2\eps \le 1+\sqrt{1+8\nu} \le 2m$.

Let us now focus on the case $K=T_{p,q}$, so that $2\nu = (p-1)(q-1)$. 
We can recast the inequalities above in terms of the semigroup function $\Gamma$ that associates to an integer $i$ the $i$--th element $\Gamma(i)$ of the semigroup $\Gamma_{p,q}$ generated by $p$ and $q$. 
Since $q<p$, we always have $\Gamma(1) = 0$ and $\Gamma(2) = q$. Manipulating the identity $d(S^3_{m^2}(T_{p,q}),j) = 0$ as in \cite[Section 6]{BCG} yields, for each $j\in\{0,\dots,m-2\}$, the two inequalities
\begin{equation}\label{e:star}
\Gamma\left(\frac{(j+1)(j+2)}2\right) \le jm+\eps
\end{equation}
and
\begin{equation}\label{e:starstar}
\Gamma\left(\frac{(j+1)(j+2)}2+1\right) > jm+\eps
\end{equation}

These are formally very similar to the inequalities $(\star_j)$ and $(\star\star_j)$ in \cite[Section 6.1]{BCG}.
In fact, the two latter inequalities collapse to an identity for $g=0$ (as in \cite{BorodzikLivingston}), and these identity are sufficient to classify singularities with one Puiseux pairs that appear as the only singularity of a rational plane curve, up to finitely many exceptions \cite[Theorem 1.1]{55letters} (see \cite[Remark 1.5]{BCG}).
It is natural to ask whether \eqref{e:star} and \eqref{e:starstar} are strong enough to provide the same kind of result in the topological, rather than in the complex curve, setting.

\begin{quest}
Are \eqref{e:star} and \eqref{e:starstar} sufficient to recover all but finitely many triples $(p,q;m^2)$ for which $S^3_{m^2}(T_{p,q})$ bounds a rational homology balls?
\end{quest}

%% file: torus.tex
\section{Integral surgeries on torus knots}\label{torus}

The goal of this section is the proof of  \cref{t:kq+1}.
Along the way, we briefly recall the plumbing description for Seifert fibred manifolds and a lemma on rational homology cobordant Seifert manifolds. We also recall Owens and Strle's concordance invariant $m$, and we describe the plumbing graph of the Seifert fibred spaces which arise as Dehn surgery on torus knots. 

Suppose $\Gamma$ is a negative canonical plumbing graph with $n$ vertices. Let $-\Gamma$ be the graph with the opposite sign on each weight.
Then $\partial P(-\Gamma)=-\partial P(\Gamma)$ and $-\Gamma$ is the positive canonical plumbing graph of $-\partial P(\Gamma)$.
The dual of $\Gamma$, i.e. $\Gamma^*$, is the negative canonical plumbing graph of $-\partial P(\Gamma)$. Finally $-\Gamma^*$
is the positive canonical plumbing graph of $\partial P(\Gamma)$. Note that $-(\Gamma^*)=(-\Gamma)^*$.
If $\Gamma$ is negative definite then $\Gamma$ embeds in $(\Z^n,-I)$ if and only if $-\Gamma$ embeds in $(\Z^n,I)$.

In order to simplify the notation and to make the proofs easier to read, in this section we switch from negative plumbing graphs to positive ones. 
Since the topological property we want to detect is not affected by reversal of orientation this modification will not affect the arguments in our proofs.

\subsection{Preliminaries}\label{ss:seifert}
We recall from \cite{NeumannRaymond} the basic definitions of Seifert manifolds.
A closed Seifert fibred manifold over $S^2$ 
can be described by its \emph{unnormalized Seifert invariants}
\[
\left(b;\frac{\alpha_1}{\beta_1},\dots,\frac{\alpha_k}{\beta_k}\right).
\]
Here $b,\alpha_i,\beta_i$ are integers, $\alpha_i> 1$ and $\gcd(\alpha_i,\beta_i)=1$.
We denote such a manifold with $Y(b;\frac{\alpha_1}{\beta_1},\dots,\frac{\alpha_k}{\beta_k})$.
These data determine the manifold, but every Seifert fibred manifold admits several such descriptions. 
A surgery description for $Y(b;\frac{\alpha_1}{\beta_1},\dots,\frac{\alpha_k}{\beta_k})$ is depicted in \cref{seifertsurgery}.

\begin{thm}\label{seiferttheorem}
 Let $\Gamma$ be the following star-shaped positive canonical plumbing graph.
\[
\xygraph{
!{<0cm,0cm>;<1cm,0cm>:<0cm,1cm>::}
!~-{@{-}@[|(2.5)]}
!{(0,1.5) }*+{\bullet}="x"
!{(1.5,3) }*+{\bullet}="a1"
!{(4.5,3) }*+{\bullet}="a2"
!{(1.5,0) }*+{\bullet}="c1"
!{(4.5,0) }*+{\bullet}="c2"
!{(1.5,2) }*+{\bullet}="b1"
!{(4.5,2) }*+{\bullet}="b2"
!{(3,0) }*+{\dots}="cm"
!{(3,3) }*+{\dots}="am"
!{(3,2) }*+{\dots}="bm"
!{(3,1) }*+{\vdots}
!{(0,1.9) }*+{b}
!{(1.5,3.4) }*+{a_1^1}
!{(4.5,3.4) }*+{a_{n_1}^1}
!{(1.5,2.4) }*+{a_1^2}
!{(4.5,2.4) }*+{a_{n_2}^2}
!{(1.5,0.4) }*+{a_1^k}
!{(4.5,0.4) }*+{a_{n_k}^k}
"x"-"c1"
"x"-"a1"
"x"-"b1"
"a1"-"am"
"b1"-"bm"
"c1"-"cm"
"a2"-"am"
"b2"-"bm"
"c2"-"cm"
}
\]
Then $\partial P(\Gamma)$ is orientation-preserving diffeomorphic to 
$Y(b;\frac{\alpha_1}{\beta_1},\dots,\frac{\alpha_k}{\beta_k})$ where
$\frac{\alpha_i}{\beta_i}=[a^i_1,\dots,a^i_{n_i}]^-$.
\end{thm}

\begin{figure}[h]
\labellist
\pinlabel $b$ at 183 46
\pinlabel $\dots$ at 110 8
\pinlabel $\frac{\alpha_1}{\beta_1}$ at 10 8
\pinlabel $\frac{\alpha_2}{\beta_2}$ at 50 8
\pinlabel $\frac{\alpha_k}{\beta_k}$ at 168 8
\endlabellist
\centering
\includegraphics[scale=1.0]{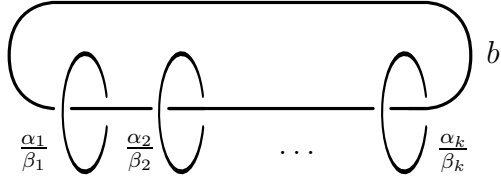}
\caption{A surgery description for the Seifert fibred manifold $Y(b;\frac{\alpha_1}{\beta_1},\dots,\frac{\alpha_k}{\beta_k})$.}
\label{seifertsurgery}
\end{figure}

The following lemma is implicit in \cite[Section 3.1]{Lecuona}. 
Here we deduce it as a special case of a more general result which can be found in \cite[Section 4]{Paolo}.

\begin{lemma}\label{join}
Let $Y:=Y(b;\frac{\alpha_1}{\beta_1},\dots,\frac{\alpha_n}{\beta_n})$ be a Seifert manifold with $\alpha_i>\beta_i\geq 1$ 
for each $i$. Assume that there exist $h,k$ such that 
\[
\frac{\beta_h}{\alpha_h}+\frac{\beta_k}{\alpha_k}=1.
\]
Let $Y':=Y(b-1;\frac{\alpha_1}{\beta_1},\dots,\widehat{\frac{\alpha_h}{\beta_h}},\dots, \widehat{\frac{\alpha_k}{\beta_k}},\dots,\frac{\alpha_n}{\beta_n})$. Then, $Y$ and $Y'$ are rational homology cobordant.
\end{lemma}

\begin{proof}
This Lemma is a direct consequence of \cite[Proposition 4.6]{Paolo}. In order to apply this result we need to show that
the positive canonical plumbing graph
of $Y$ is obtained from that of $Y'$ by
attaching a plumbing graph which represents $S^1\times S^2$. 
The attachment is obtained by identifing two vertices, the weight of the new vertex is the sum of the weitghts of the identified vertices.

 Consider the linear graph
\[
\xygraph{
!{<0cm,0cm>;<1cm,0cm>:<0cm,1cm>::}
!~-{@{-}@[|(2.5)]}
!{(0,0) }*+{\bullet}="x"
!{(1.5,0) }*+{\dots}="a1"
!{(3,0) }*+{\bullet}="a2"
!{(4.5,0) }*+{\bullet}="c1"
!{(6,0) }*+{\bullet}="d"
!{(7.5,0) }*+{\dots}="e"
!{(9,0) }*+{\bullet}="f"
!{(0,0.4) }*+{a_n}
!{(3,0.4) }*+{a_1}
!{(4.5,0.4) }*+{1}
!{(6,0.4) }*+{b_1}
!{(9,0.4) }*+{b_m}
"x"-"a1"
"a2"-"a1"
"a2"-"c1"
"c1"-"d"
"d"-"e"
"e"-"f"
}
\]
where $a_i\geq 2$, $b_j\geq 2$ for each $i,j$ and 
\[
[a_1,\dots,a_n]^-=\frac{\alpha_h}{\beta_h}; \qquad [b_1,\dots,b_m]^-=\frac{\alpha_k}{\beta_k}.
\]

It is easy to see that this plumbing graph describes the plumbed $3$--manifold $S^1\times S^2$. 
It follows from \cref{seiferttheorem} that  the canonical plumbing graph of $Y$ is obtained from that of $Y'$ by adding two legs with weights 
$(a_1,\dots,a_n)$,$(b_1,\dots,b_m)$ and by increasing the weight of the central vertex by one. This operation is precisely the attachment described above and the conclusion now follows from
\cite[Proposition 4.6]{Paolo}.
\end{proof}

Observe that, in order to apply \cref{join}, one needs $\alpha_h = \alpha_k$, i.e. that the two fibres must have the same multiplicity.

Given a knot $K\subset S^3$ Owens and Strle in \cite{OwensStrle} introduce the following invariant:
\[
m(K)=\textrm{inf}\{r\in\mathbb{Q}_{>0}\mid S^3_r(K) \textrm{ bounds a negative definite 4--manifold}\}.
\]
They show that $m$ is a concordance invariant and that it vanishes on negative knots. 
They also show the following
\begin{thm}\label{OwTh}
 Let $T_{p,q}$ be the positive $(p,q)$--torus knot. Then 
 \[
 m(T_{p,q})=\left\{
 \begin{array}{ll}
  pq-\frac{q}{p*}& \textrm{if }$n$\textrm{ is even}\\
  pq-\frac{p}{q*}& \textrm{if }$n$\textrm{ is odd}
 \end{array}
 \right.
 \]
 where 
 \begin{itemize}
  \item $n$ is the number of steps in the Euclidean algorithm for $p$ and $q$,
  \item $0<q^*<p$ is such that $qq^*\equiv 1 \pmod p$,
  \item $0<p^*<q$ is such that $pp^*\equiv 1 \pmod q$.
 \end{itemize} 
\end{thm}

\begin{rem}
 Since rational homology balls are negative, definite, if $S^3_r(K)$ bounds a $\Q HB^4$
 for some $r>0$ then, $r\geq m(K)$.
\end{rem}

We are now in position to give the second proof of \cref{t:pvs9q}.

\begin{proof}[Second proof of \cref{t:pvs9q}]
It follows from \cref{OwTh} that $m(T_{p,q}) \ge q(p-1)$;
 therefore, by the remark above, if $m^2<q(p-1)$ then $S^3_{m^2}(T_{p,q})$ does not bound a rational homology ball. 

Applying the right-hand side inequality of \cref{p:mvsnu}, we get that if
\[
\left(\frac{3+\sqrt{9+8\nu}}2\right)^2 \le q(p-1),
\]
no integral surgery along $T_{p,q}$ bounds a rational homology ball. We now set out to prove that if $p>9q$ the above inequality holds. Indeed, expanding and substituting $2\nu = (p-1)(q-1)$ yields
\[
3\sqrt{9+8\nu} \le 2p-11
\]
Since $p>9q\ge 18$, we can square again and obtain
\[81+36pq-36p-36q+36 \le 4p^2-44p+121 \Longleftrightarrow 9(p-1)q \le p^2-2p-1 \Longleftrightarrow 9q + 1 \le p,
\]
which is exactly what we required.
\end{proof}

The following lemma is proved in \cite[Lemma 4.4]{OwensStrle}.
\begin{lemma}\label{Ow}
 For any rational number $r$
\[
 S^3_r(T_{p,q})=Y\left(2;\frac{p}{q^*},\frac{q}{p^*},\frac{pq-r}{pq-r-1}\right).\qedhere
 \]
 \end{lemma}
Now we give an explicit description of the canonical plumbing graphs 
which describe certain surgeries on torus knots. 

\begin{prop}\label{r>pq+1}
 Assume that $r\notin [pq-1,pq+1]$. The $3$--manifold $S^3_r(T_{p,q})$ can be described via the 
 following positive canonical plumbing graph
\[
\xygraph{
!{<0cm,0cm>;<1cm,0cm>:<0cm,1cm>::}
!~-{@{-}@[|(2.5)]}
!{(0,1) }*+{\bullet}="x"
!{(1.5,2) }*+{\bullet}="a1"
!{(4.5,2) }*+{\bullet}="a2"
!{(1.5,0) }*+{\bullet}="c1"
!{(4.5,0) }*+{\bullet}="c2"
!{(1.5,1) }*+{\bullet}="b1"
!{(4.5,1) }*+{\bullet}="b2"
!{(3,0) }*+{\dots}="cm"
!{(3,2) }*+{\dots}="am"
!{(3,1) }*+{\dots}="bm"
!{(0,1.4) }*+{d}
!{(1.5,2.4) }*+{a_1}
!{(4.5,2.4) }*+{a_n}
!{(1.5,1.4) }*+{b_1}
!{(4.5,1.4) }*+{b_m}
!{(1.5,0.4) }*+{c_1}
!{(4.5,0.4) }*+{c_k}
"x"-"c1"
"x"-"a1"
"x"-"b1"
"a1"-"am"
"b1"-"bm"
"c1"-"cm"
"a2"-"am"
"b2"-"bm"
"c2"-"cm"
}
\]
where $[a_1,\dots,a_n]^-=\frac{p}{q^*}$, $[b_1,\dots,b_m]^-=\frac{q}{p^*}$, and:
 \begin{itemize}
  \item if $r > pq+1$ then $d=1$ and $[c_1,\dots,c_k]^-=r-pq$; 
  \item if $r < pq-1$ then $d=2$ and $[c_1,\dots,c_k]^-=\frac{pq-r}{pq-r-1}$. 
\end{itemize}
\end{prop}
\begin{proof}
 Note that $\frac{p}{q^*}>1$ and $\frac{q}{p^*}>1$. Moreover
\[
\frac{pq-r}{pq-r-1}=\frac{r-pq}{r-pq+1}=1-\frac{1}{r-pq+1}.
\]
If $r-pq>1$, by \cref{Ow} and \cref{seifertsurgery}
we see that $S^3_r(T_{p,q})$ is described by the positive canonical plumbing graph depicted above. The case
$r<pq-1$ is analogous.
\end{proof}
\subsection{Main results}

A \emph{$2$--chain} is a linear connected plumbing graph where the weight of every vertex is $2$.
\begin{lemma}\label{2chain}
Let $\Gamma$ be a positive definite plumbing graph with $m$ vertices.
 Suppose that there exist vertices
$v_1,\dots,v_k$ such that $\Gamma\setminus\{v_1,\dots,v_k\}=\Gamma_1\sqcup\dots\sqcup\Gamma_h$,
where each $\Gamma_i$ is a $2$--chain
 whose length is not $3$, there is at most one $2$--chain of length $1$, and $h>k$. Then $\Gamma$ does not embed.
\end{lemma}
\begin{proof}
Every 2--chain whose length is not $3$ has an essentially unique realisation as a sublattice of the standard
lattice $(\mathbb{Z}^N,I)$. In particular $N$ is strictly
larger than length of the 2--chain.
If, moreover, there are no chains of length $1$ it is easy to see that
each basis vector, say $e_i$, appears in the expression of a unique $2$--chain. Therefore every $2$--chain requires
a set of basis vectors which are orthogonal to each vector belonging to any another $2$--chain.
It follows that the subgraph $\widetilde{\Gamma}:=
\Gamma\setminus\{v_1,\dots,v_k\}=\Gamma_1\sqcup\dots\sqcup\Gamma_h$
has a unique realization as a sublattice of the standard
lattice $(\mathbb{Z}^N,I)$ and that $N\geq m-k+h>m$.
\end{proof}

Following~\cite{Lisca-ribbon}, we use the notation $[\dots,2^{[h]},\dots]^-$ to denote a length-$h$ string of 2 in a continued fraction expansion.
\begin{lemma}\label{confrac}
 For each $q\geq 2$ and $k\geq 1$ we have
\begin{enumerate}
 \item $\frac{kq+1}{kq+1-k}=[2^{[q-1]},k+1]^-$
 \item $\frac{kq^2+q+1}{q^2}=[k+1,2^{[q-2]},2+q]^-$
 \item $\frac{kq^2+q-1}{q^2}=[k+1,2^{[q]},q]^-$
 \item $\frac{q^2}{q-1}=[q+2,2^{[q-2]}]^-$
 \item $\frac{q^2}{q+1}=[q,2^{[q]}]^-$
 \end{enumerate}
\end{lemma}

These identities can easily be proved by induction. We omit the proof.

Now we are ready to prove \cref{t:kq+1}.
\begin{proof}[Proof of \cref{t:kq+1}]
We start by dealing with two special cases, namely $q=2$ and $k=1$.

First assume that $q=2$. In this case $T_{p,q}$ is alternating and therefore we can apply \cref{p:alternating}. If $S^3_{n}(T_{p,q})$ bounds a $\mathbb{Q}HB^4$ then $n\leq 25$.
Moreover, since $p-1=2\nu^+(T_{p,2})=-\sigma(T_{p,2}) = -\sigma<0$ we have:
  
\begin{itemize}
\item if $n=1$ then $\sigma\ge0$ and we can ignore this possibility because $T_{p,2}$ is a positive torus knot;
\item if $n=4$ then $\sigma = -2$, that is $p=3$ and $S^3_4(T_{3,2})=Y(2;\frac{3}{2},2,2)$ bounds a $\mathbb{Q}HB^4$ by \cref{join};
\item if $n=9$ then $p\in\{3,5\}$.
The $3$--manifold $S^3_9(T_{3,2})=Y(2;\frac{3}{2},2,\frac{3}{2})$ bounds a $\mathbb{Q}HB^4$ by \cite{55letters} while $S^3_9(T_{5,2})=-L(9,4)$ bounds a $\mathbb{Q}HB^4$ by \cite{Lisca-ribbon};
\item if $n=16$ then $p\in\{7,9\}$. By \cref{join} both $S^3_{16}(T_{7,2})$ and $S^3_{16}(T_{9,2})$ are rational homology cobordant to $-L(4,1)$ and therefore each bounds a rational homology ball;
\item if $n=25$ then $p=13$ and $S^3_{25}(T_{13,2})=-L(25,4)$ which bounds a rational homology ball by \cite{Lisca-ribbon}.
\end{itemize}

Now assume that $k=1$. It is easy to see, via \cref{Ow}, that 
 \[
 S^3_n(T_{q+1,q})=Y\left(2;\frac{q+1}{q},q,\frac{q^2+q-n}{q^2+q-n-1}\right).
 \]
If, in particular, $n=q^2$ we have $S^3_{q^2}(T_{q+1,q})=Y\left(2;\frac{q+1}{q},q,\frac{q}{q-1}\right)$.
By \cref{join} this manifold is $\Q H$-cobordant to $S^3$.
Since $m(T_{q+1,q})=q^2$, if $T_{q+1,q}$ has another integral surgery that bounds a $\mathbb{Q}HB^4$ then, by \cref{p:mvsnu1}, this manifold is $S^3_{(q+1)^2}(T_{q+1,q})$.
We have
\[
 S^3_{(q+1)^2}(T_{q+1,q})=Y\left(2;\frac{q+1}{q},q,\frac{q+1}{q+2}\right)=Y\left(1;\frac{q+1}{q},q,q+1\right)
\]
 which, again by \cref{join}, bounds a $\mathbb{Q}HB^4$.
 
 From now on we may assume that $k>1$ and $q>2$. Write $p=kq\pm 1$ and assume that $S^3_n(T_{p,q})$ bounds a $\mathbb{Q}HB^4$.
 We split the proof in two cases according to whether $p=kq+1$ or $p=kq-1$. Each case is further divided into subcases distinguishing the possible values for the surgery coefficient $n$.\\ 
\vskip 0.5 cm
 \textbf{First case}: $p=kq+1$.\\
 Note that $n\geq m(T_{kq+1,q})=kq^2$.
 We distinguish the following three cases
\begin{enumerate}
\item  $n>kq^2+q+1$;
 \item $kq^2\leq n<kq^2+q-1$;
 \item $n\in\{kq^2+q-1,kq^2+q,kq^2+q+1\}$.
\end{enumerate}

\vskip 0.5 cm

\textbf{First subcase}: $n>kq^2+q+1$.\\
   We claim that there are no candidate triples in this subcase. It is enough to show that:
 \begin{equation}\label{boh}
   n\leq kq^2+q+1.
 \end{equation}
 
Note that inequality
\eqref{boh} implies that there is at most one value of $n$ such that the corresponding surgered 
manifold bounds a $\mathbb{Q}HB^4$. To see this we need to show that the interval 
$[kq^2,kq^2+q+1]$ contains at most one square. Suppose $N^2$ is the smallest square in $[kq^2,kq^2+q+1]$. 
We have 
\[
kq^2\leq N^2\Rightarrow q\leq \frac{N}{\sqrt{k}}
\]
and therefore
\[
kq^2+q+1\leq N^2+\frac{N}{\sqrt{k}}+1<N^2+2N+1=(N+1)^2.
\]
 To prove \eqref{boh} we consider the positive canonical plumbing graph associated to $-S^3_n(T_{kq+1,q})$ with $n>kq^2+q+1=pq+1$, which is easily seen to be positive definite.
This plumbing graph is obtained by taking the dual of the graph described in \cref{r>pq+1}, and it has the form
\[
\xygraph{
!{<0cm,0cm>;<1cm,0cm>:<0cm,1cm>::}
!~-{@{-}@[|(2.5)]}
!{(0,1.5) }*+{\bullet}="x"
!{(1.5,3) }*+{\bullet}="a1"
!{(4.5,3) }*+{\bullet}="a2"
!{(1.5,0) }*+{\bullet}="c1"
!{(4.5,0) }*+{\bullet}="c2"
!{(1.5,1.5) }*+{\bullet}="b1"
!{(3,1.5) }*+{\bullet}="b2"
!{(6,1.5) }*+{\bullet}="b3"
!{(3,0) }*+{\dots}="cm"
!{(3,3) }*+{\dots}="am"
!{(4.5,1.5) }*+{\dots}="bm"
!{(4.5,2.5)}*+{\overbrace{\hphantom{--------}}^{k-1}}
!{(1.5,1.9) }*+{q+1}
!{(3,1.9) }*+{2}
!{(6,1.9) }*+{2}
!{(0,1.9) }*+{2}
!{(1.5,3.4) }*+{2}
!{(4.5,3.4) }*+{2}
!{(1.5,0.4) }*+{2}
!{(4.5,0.4) }*+{2}
!{(3,1) }*+{\overbrace{\hphantom{--------}}^{n-kq^2-q-1}}
!{(3,4) }*+{\overbrace{\hphantom{--------}}^{q-1}}
"x"-"c1"
"x"-"a1"
"x"-"b1"
"b1"-"b2"
"a1"-"am"
"b2"-"bm"
"c1"-"cm"
"a2"-"am"
"b3"-"bm"
"c2"-"cm"
}
\]

By removing the vertex whose weight is $q+1$ we obtain a disjoint union of two 2--chains. One of these 2--chains
has length 

\[
(q-1)+1+(n-kq^2-q-1)\geq 4
\]
since $q\geq 3$ and $n>kq^2+q+1$, and the other has length $k-1$. 
If $k\neq4$, it follows by \cref{2chain} that the above 
plumbing graph does not embed and therefore $S^3_n(T_{kq+1,q})$ does not bound a $\mathbb{Q}HB^4$.
If $k=4$ the plumbing graph has the form
\[
\xygraph{
!{<0cm,0cm>;<1cm,0cm>:<0cm,1cm>::}
!~-{@{-}@[|(2.5)]}
!{(0,1.5) }*+{\bullet}="x"
!{(1.5,3) }*+{\bullet}="a1"
!{(4.5,3) }*+{\bullet}="a2"
!{(1.5,0) }*+{\bullet}="c1"
!{(4.5,0) }*+{\bullet}="c2"
!{(1.5,1.5) }*+{\bullet}="b1"
!{(3,1.5) }*+{\bullet}="b2"
!{(6,1.5) }*+{\bullet}="b3"
!{(3,0) }*+{\dots}="cm"
!{(3,3) }*+{\dots}="am"
!{(4.5,1.5) }*+{\bullet}="bm"
!{(1.5,1.9) }*+{q+1}
!{(3,1.9) }*+{2}
!{(6,1.9) }*+{2}
!{(4.5,1.9) }*+{2}
!{(0,1.9) }*+{2}
!{(1.5,3.4) }*+{2}
!{(4.5,3.4) }*+{2}
!{(1.5,0.4) }*+{2}
!{(4.5,0.4) }*+{2}
!{(3,1) }*+{\overbrace{\hphantom{--------}}^{n-4q^2-q-1}}
!{(3,4) }*+{\overbrace{\hphantom{--------}}^{q-1}}
"x"-"c1"
"x"-"a1"
"x"-"b1"
"b1"-"b2"
"a1"-"am"
"b2"-"bm"
"c1"-"cm"
"a2"-"am"
"b3"-"bm"
"c2"-"cm"
}
\]

  Consider the following portion of the above plumbing graph

\[
\xygraph{
!{<0cm,0cm>;<1cm,0cm>:<0cm,1cm>::}
!~-{@{-}@[|(2.5)]}
!{(0,0) }*+{\bullet}="x"
!{(1.5,0) }*+{\bullet}="a1"
!{(3,0) }*+{\bullet}="a2"
!{(4.5,0) }*+{\bullet}="c1"
!{(0,0.4) }*+{q+1}
!{(1.5,0.4) }*+{2}
!{(3,0.4) }*+{2}
!{(4.5,0.4) }*+{2}
"x"-"a1"
"a2"-"a1"
"a2"-"c1"
}
\]
There are two choices for the embedding of the 2--chain. One choice requires four basis vectors
and can be excluded (as in the proof of \cref{2chain}). The other choice gives us 
an embedding of the 2--chain of the form
\[
e_1+e_2,e_2+e_3,e_2-e_1.
\]
It is easy to see that there is no vector $v$ such that $v\cdot(e_1+e_2)=1$ and $v\cdot(e_2-e_1)=0$.
This shows that $n\leq kq^2+q+1$ and inequality \eqref{boh} follows.

\vskip 0.5 cm
\textbf{Second subcase}: $kq^2\leq n<kq^2+q-1$.\\
Suppose that $k=4$. We have
\[
S^3_{4q^2}(T_{4q+1,q})=Y\left(2;\frac{4q+1}{4q-3},q,\frac{q}{q-1}\right)
\]
By \cref{join}, applied to the two exceptional fibres with invariants $q$ and $\frac{q}{q-1}$, this manifold is $\mathbb{Q}H$--cobordant to a lens space obtained 
via the following plumbing graph
\[
\xygraph{
!{<0cm,0cm>;<1cm,0cm>:<0cm,1cm>::}
!~-{@{-}@[|(2.5)]}
!{(0,0) }*+{\bullet}="x"
!{(1.5,0) }*+{\bullet}="a1"
!{(3,0) }*+{\dots}="a2"
!{(4.5,0) }*+{\bullet}="c1"
!{(0,0.4) }*+{1}
!{(1.5,0.4) }*+{a_1}
!{(4.5,0.4) }*+{a_n}
"x"-"a1"
"a2"-"a1"
"a2"-"c1"
}
\]
where $[a_1,\dots,a_n]^-=\frac{4q+1}{4q-3}$. It is easy to see that this lens space is just 
$-L(4,1)$. It follows that $S^3_{4q^2}(T_{4q+1,q})$ bounds a $\mathbb{Q}HB^4$ and, by~\eqref{boh} this is the only
integral surgery on $T_{4q+1,q}$ with this property.

Recall from \cref{t:pvs9q} that $k<9$.
Now we will examine each possible value of $k$ in the set $\{2,3,5,6,7,8\}$.
We have  
\[
S^3_n(T_{kq+1,q})=Y\left(2;\frac{kq+1}{kq+1-k},q,\frac{kq^2+q-n}{kq^2+q-n-1}\right)
\]
and all the Seifert
invariants are strictly greater than 1. It follows that the positive canonical plumbing graph
for $S^3_n(T_{kq+1,q})$ has the form

\[
\xygraph{
!{<0cm,0cm>;<1cm,0cm>:<0cm,1cm>::}
!~-{@{-}@[|(2.5)]}
!{(0,1.5) }*+{\bullet}="x"
!{(1.5,3) }*+{\bullet}="a1"
!{(4.5,3) }*+{\bullet}="a2"
!{(6,3) }*+{\bullet}="a3"
!{(1.5,0) }*+{\bullet}="c1"
!{(4.5,0) }*+{\bullet}="c2"
!{(1.5,1.5) }*+{\bullet}="b1"
!{(3,0) }*+{\dots}="cm"
!{(3,3) }*+{\dots}="am"
!{(1.5,1.9) }*+{q}
!{(6,3.4) }*+{k+1}
!{(0,1.9) }*+{2}
!{(1.5,3.4) }*+{2}
!{(4.5,3.4) }*+{2}
!{(1.5,0.4) }*+{2}
!{(4.5,0.4) }*+{2}
!{(3,1) }*+{\overbrace{\hphantom{--------}}^{kq^2+q-n-1}}
!{(3,4) }*+{\overbrace{\hphantom{--------}}^{q-1}}
"x"-"c1"
"x"-"a1"
"x"-"b1"
"a1"-"am"
"c1"-"cm"
"a2"-"am"
"a2"-"a3"
"c2"-"cm"
}
\]
where we have used \cref{confrac}, i.e. that $\frac{kq+1}{kq-1+k}=[2^{[q-1]},k+1]^-$.
In order to emphasise the 2--chain in this plumbing graph we rewrite it as

\[
\xygraph{
!{<0cm,0cm>;<1cm,0cm>:<0cm,1cm>::}
!~-{@{-}@[|(2.5)]}
!{(4.5,1.5) }*+{\bullet}="x"
!{(0,1.5) }*+{\bullet}="a1"
!{(3,1.5) }*+{\bullet}="a2"
!{(6,1.5) }*+{\bullet}="c1"
!{(7.5,1.5) }*+{\dots}="cm"
!{(9,1.5) }*+{\bullet}="c2"
!{(10.5,1.5) }*+{\bullet}="c3"
!{(1.5,1.5) }*+{\dots}="am"
!{(4.5,0) }*+{\bullet}="b1"
!{(4.5,-0.4) }*+{q}
!{(6,1.9) }*+{2}
!{(0,1.9) }*+{2}
!{(3,1.9) }*+{2}
!{(4.5,1.9) }*+{2}
!{(6,1.9) }*+{2}
!{(9,1.9) }*+{2}
!{(10.5,1.9) }*+{k+1}
!{(7.5,2.5) }*+{\overbrace{\hphantom{--------}}^{q-1}}
!{(1.5,2.5) }*+{\overbrace{\hphantom{--------}}^{kq^2+q-n-1}}
"x"-"c1"
"x"-"a2"
"x"-"b1"
"a1"-"am"
"c1"-"cm"
"a2"-"am"
"c2"-"cm"
"c2"-"c3"
}
\]
Write $N=kq^2+2q-n-1$ for the length of the 2--chain. Note that $N\geq 4$.
The graph has $N+2$ vertices. 
If the corresponding integral lattice embeds in $(\mathbb{Z}^{N+2},I)$ then we can write
the vectors corresponding to the 2--chain as
\[
e_1+e_2,\dots,e_{kq^2+q-n-1}+e_{kq^2+q-n},e_{kq^2+q-n}+e_{kq^2+q-n+1},\dots,e_N+e_{N+1}.
\]
Let us call $v$ the vector whose weight is $k+1$. It can be written in one of the two following ways:
\begin{itemize}
\item $v=\alpha e_{N+2}+\beta e_{N+1}+\gamma \sum_{i=1}^{N}(-1)^{i}e_i$, with $\gamma\neq 0$;
\item $v=e_{N+1}+\sqrt{k}e_{N+2}$.
\end{itemize}
The second equality implies that $k$ is a square and can be excluded (since we already dealt with 
the values $1$ and $4$ and we know that $k<9$). Therefore the first equation holds. 
Since $v\cdot v=k+1$ we have
\[
\alpha^2+\beta^2+N\gamma^2=k+1\leq 9.
\]
Since $v\cdot (e_N+e_{N+1})=1$ we see that $\beta+\gamma=1$. Then, we may write
$k+1=\alpha^2+\beta^2+N(1-\beta)^2$. Moreover, since $N\geq 4$, we obtain
\[
\alpha^2+\beta^2+4(1-\beta)^2\leq \alpha^2+\beta^2+N(1-\beta)^2\leq 9
\]
from which we get $(1-\beta)^2\leq 1$. Since $\gamma\neq 0$ and  $\beta+\gamma=1$ we see that $\beta\neq 1$.
We conclude that $(1-\beta)^2= 1$, i.e. $\beta\in\{0,2\}$. 

If $\beta=2$ we have 
\[
8\leq \alpha^2+4+N\leq 9
\]
Which implies $(\alpha,k)\in\{(0,7),(0,8),(\pm 1,8)\}$. Note that $N\ge q+1$.
In the first case we quickly obtain 
$(p,q;n)=(22,3;64)$ and by \cite{55letters} the corresponding surgered manifold bounds a $\mathbb{Q}HB^4$.
In the second case we have $n\in\{72,126\}$, and thus it is not a square.
In the third case we get $n=73$ which is not a square. 

If $\beta=0$ we have
\[
\alpha^2+4\leq \alpha^2+N\leq 9
\]
In particular $|\alpha|\leq 2$. If $\alpha=0$ we have $N=k+1$ and $n=kq^2+2q-k-2$. Notice that since $N\ge q+1$ and $k < 9$ there are only finitely many possible values of $q$.

By listing all possibilities, we obtain the set of all candidate triples $(p,q;n)$. The only one for which $n$ is a square is $(21,4;64)$. The corresponding definite plumbing graph is positive definite:
\[
\xygraph{
!{<0cm,0cm>;<1cm,0cm>:<0cm,1cm>::}
!~-{@{-}@[|(2.5)]}
!{(0,1.5) }*+{\bullet}="x"
!{(1.5,3) }*+{\bullet}="a1"
!{(4.5,3) }*+{\bullet}="a2"
!{(6,3) }*+{\bullet}="a3"
!{(1.5,0) }*+{\bullet}="c1"
!{(4.5,0) }*+{\bullet}="c2"
!{(1.5,1.5) }*+{\bullet}="b1"
!{(3,0) }*+{\dots}="cm"
!{(3,3) }*+{\bullet}="am"
!{(1.5,1.9) }*+{4}
!{(3,3.4) }*+{2}
!{(6,3.4) }*+{6}
!{(0,1.9) }*+{2}
!{(1.5,3.4) }*+{2}
!{(4.5,3.4) }*+{2}
!{(1.5,0.4) }*+{2}
!{(4.5,0.4) }*+{2}
!{(3,1) }*+{\overbrace{\hphantom{--------}}^{19}}
"x"-"c1"
"x"-"a1"
"x"-"b1"
"a1"-"am"
"c1"-"cm"
"a2"-"am"
"a2"-"a3"
"c2"-"cm"
}
\]
And the corresponding integral lattice does not embed, by direct inspection.

If $|\alpha|=1$ we have $4\leq N=k\leq 9$. Just like in the previous case we can list all the candidate triples
and discard the ones whose surgery coefficient is not a square. We are left with $(43,6;256)$ which by \cite{55letters} 
corresponds to a manifold which bounds a $\mathbb{Q}HB^4$.

If $|\alpha|=2$ we have $N\in\{4,5\}$ and $k=N+3$. In this case the only triple we get is $(22,3;64)$, which, as already noted,
bounds a $\mathbb{Q}HB^4$.
\vskip 0.5 cm
\textbf{Third subcase}: $n\in\{kq^2+q-1,kq^2+q,kq^2+q+1\}$.\\
Suppose $n=pq=kq^2+q$, then we have
\[
S^3_{kq^2+q}(T_{kq+1,q})\cong -L(kq+1,q)\# -L(q,kq+1)\cong -L(kq+1,q)\# -L(q,1).
\]
It follows from \cite{Lisca-ribbon} that this manifold bounds a rational homology ball if and only if $q=4$ 
and $k\in\{2,6\}$. These
values correspond to the triples $(25,4;100)$ and $(9,4;36)$. 
Recall that
$S^3_{pq\pm 1}(T_{p,q})=-L(pq\pm1,q^2)$. 
It follows that when $n=pq\pm1=kq^2+q\pm1$ we have 
\[
S^3_{kq^2+q\pm1}(T_{kq+1,q})=-L(kq^2+q\pm 1,q^2).
\]
According to \cite{Lisca-ribbon} if this lens space bounds a $\mathbb{Q}HB^4$ then $-3\leq I(\frac{kq^2+q + 1}{q^2})\leq 1$. It follows from \cref{confrac} that $I(\frac{kq^2+q+1}{q^2})=k-1$, and therefore we may assume that $k=2$. Again by \cref{confrac}
the positive plumbing graph associated to $L(2q^2+q+1,q^2)$ can be depicted as  

\[
\xygraph{
!{<0cm,0cm>;<1cm,0cm>:<0cm,1cm>::}
!~-{@{-}@[|(2.5)]}
!{(0,0) }*+{\bullet}="a1"
!{(1.5,0) }*+{\bullet}="a2"
!{(3,0) }*+{\dots}="a3"
!{(4.5,0) }*+{\bullet}="a4"
!{(6,0) }*+{\bullet}="a5"
!{(0,0.4) }*+{3}
!{(1.5,0.4) }*+{2}
!{(4.5,0.4) }*+{2}
!{(6,0.4) }*+{2+q}
!{(3,1) }*+{\overbrace{\hphantom{--------}}^{q-2}}
"a1"-"a2"
"a3"-"a2"
"a3"-"a4"
"a5"-"a4"
}
\]
It is easy to check that the corresponding integral lattice does not embed
(start with the $2$--chain and then examine the vector whose square is $3$).
In order to conclude we need to 
the study the family of lens spaces $L(kq^2+q-1,q^2)$. By \cref{confrac} the positive plumbing graph corresponding
to $-L(kq^2+q-1,q^2)$ can be depicted as 
\[
\xygraph{
!{<0cm,0cm>;<1cm,0cm>:<0cm,1cm>::}
!~-{@{-}@[|(2.5)]}
!{(0,0) }*+{\bullet}="a1"
!{(1.5,0) }*+{\bullet}="a2"
!{(3,0) }*+{\dots}="a3"
!{(4.5,0) }*+{\bullet}="a4"
!{(6,0) }*+{\bullet}="a5"
!{(0,0.4) }*+{k+1}
!{(1.5,0.4) }*+{2}
!{(4.5,0.4) }*+{2}
!{(6,0.4) }*+{q}
!{(3,1) }*+{\overbrace{\hphantom{--------}}^{q}}
"a1"-"a2"
"a3"-"a2"
"a3"-"a4"
"a5"-"a4"
}
\]
The dual of this linear graph is
\[
\xygraph{
!{<0cm,0cm>;<1cm,0cm>:<0cm,1cm>::}
!~-{@{-}@[|(2.5)]}
!{(0,0) }*+{\bullet}="a1"
!{(1.5,0) }*+{\dots}="a2"
!{(3,0) }*+{\bullet}="a3"
!{(4.5,0) }*+{\bullet}="a4"
!{(6,0) }*+{\bullet}="a5"
!{(7.5,0) }*+{\dots}="a6"
!{(9,0) }*+{\bullet}="a7"
!{(0,0.4) }*+{2}
!{(3,0.4) }*+{2}
!{(4.5,0.4) }*+{q+3}
!{(6,0.4) }*+{2}
!{(9,0.4) }*+{2}
!{(1.5,1) }*+{\overbrace{\hphantom{--------}}^{k-1}}
!{(7.5,1) }*+{\overbrace{\hphantom{--------}}^{q-2}}
"a1"-"a2"
"a3"-"a2"
"a3"-"a4"
"a5"-"a4"
"a5"-"a6"
"a7"-"a6"
}
\]
Recall that we may assume $k\notin\{1,4\}$ and $q>2$. If $(k,q)\notin\{(2,3),(2,5)\}$ the integral lattice which corresponds to the above graph does not embed by \cref{2chain}.

If $(k,q)=(2,3)$ the surgery coefficient is $20$, and if $(k,q)=(2,5)$ it is $54$; in neither case is it a square, hence the corresponding 3--manifolds do not bound.

\subsection*{Second case: $p=kq-1$}
By \cref{t:pvs9q} we may assume that $k\leq 9$.
Note that $p^*=q-1$ and $q^*=k$. We obtain 
\[
S^3_n(T_{kq-1,q})=Y\left(2;\frac{kq-1}{k},\frac{q}{q-1},\frac{kq^2-q-n}{kq^2-q-n-1}\right).
\]
Just like in the first case we split the proof according to the following possibilities:
\begin{enumerate}
 \item $n>pq+1=kq^2-q+1$;
 \item $n<pq-1=kq^2-q-1$;
  \item $n\in\{kq^2-q-1,kq^2-q,kq^2-q+1\}$.
\end{enumerate}
\vskip 0.5 cm
\textbf{First subcase}: $n>pq+1=kq^2-q+1$.\\
The positive canonical
plumbing graph for $S^3_n(T_{kq-1,q})$, described in \cref{r>pq+1}, is indefinite.
By taking the dual of this graph, i.e. by considering the positive canonical plumbing graph of $-S^3_n(T_{kq-1,q})$, we obtain a positive definite plumbing graph. 
It can be written as

\[
\xygraph{
!{<0cm,0cm>;<1cm,0cm>:<0cm,1cm>::}
!~-{@{-}@[|(2.5)]}
!{(0,1.5) }*+{\bullet}="x"
!{(1.5,3) }*+{\bullet}="a1"
!{(3,3) }*+{\dots}="am"
!{(4.5,3) }*+{\bullet}="a2"
!{(6,3) }*+{\bullet}="a3"
!{(7.5,3) }*+{\bullet}="a4"
!{(9,3) }*+{\dots}="am2"
!{(10.5,3) }*+{\bullet}="a5"
!{(1.5,0) }*+{\bullet}="c1"
!{(3,0) }*+{\dots}="cm"
!{(4.5,0) }*+{\bullet}="c2"
!{(3,3) }*+{\dots}="am"
!{(1.5,1.5) }*+{\bullet}="b1"
!{(1.5,1.9) }*+{q}
!{(1.5,3.4) }*+{2}
!{(4.5,3.4) }*+{2}
!{(6,3.4) }*+{3}
!{(7.5,3.4) }*+{2}
!{(10.5,3.4) }*+{2}
!{(1.5,0.4) }*+{2}
!{(0,1.9) }*+{2}
!{(4.5,0.4) }*+{2}
!{(3,1) }*+{\overbrace{\hphantom{--------}}^{n-kq^2+q-1}}
!{(3,4) }*+{\overbrace{\hphantom{--------}}^{q-2}}
!{(9,4) }*+{\overbrace{\hphantom{--------}}^{k-2}}
"x"-"c1"
"x"-"a1"
"x"-"b1"
"a1"-"am"
"c1"-"cm"
"a2"-"am"
"c2"-"cm"
"a2"-"a3"
"a3"-"a4"
"a4"-"am2"
"a5"-"am2"
}
\]
This graph has two $2$--chains: the first has length $N_1=(n-kq^2+q-1) + 1 + (q-2)$, which is at least 3 since all summands are positive; the second has length $N_2=k-2$.

Observe that $N_1=3$ only if $q=3$ and $n=pq+2$, in which case $n$ can never be a square since 2 is not a quadratic residue mod 3. We can then assume that $N_1>3$, and that the vectors in the corresponding chain are
  \[
  e_1+e_2,\dots,e_{n-kq^2+q}+e_{n-kq^2+q+1},\dots,e_{N_1}+e_{N_1+1}
  \]
  where  $e_{n-kq^2+q}+e_{n-kq^2+q+1}$ corresponds to the central vertex of the plumbing graph.
As in the proof of \cref{2chain}, the second $2$--chain can be written as
  \[
  e_{N_1+2}+e_{N_1+3},\dots,e_{N_1+N_2+1}+e_{N_1+N_2+2}
  \]
  since it is easy to exclude the exceptional embedding $\{e_1+e_2,e_1+e_3,e_1-e_2\}$ when $N_2=3$. 
  Note that in the above expressions for the $2$--chains we already used all the available basis vectors.
  Call $v$ the vector whose square is $3$. Clearly $v$ must hit at least two basis vectors which appear in the same 2--chain.
  If $v$ hits more than one basis vector appearing in a 2--chain then it hits all the basis vectors in that chain. It follows that $N_2=2$, i.e. $k=4$.
  
Call $w$ the vector whose square is $q$. The top portion of the $2$--chain linked to $w$ involves $q-1$ basis vectors. It follows that $w$ must hit some basis vector in the bottom portion of $2$--chain, and therefore all of them.
This implies that $w$ does not hit any basis vector in the top portion of the two chain (we would have $w\cdot w>q$). Therefore $n-kq^2+q=q$, i.e. $n=4q^2$.
  The surgered manifold is 
  \[
  Y\left(2;\frac{4q-1}{4q-5},q,\frac{q}{q-1}\right).
  \]
  By \cref{join} these manifolds are rational homology cobordant to $L(4,1)$, and therefore to $S^3$. 
  They correspond to the triples
  $(4q-1,q;4q^2)$.
  \vskip 0.5 cm

\textbf{Second subcase}: $n<pq-1$.\\
The positive canonical plumbing graph of $S^3_n(T_{kq-1,q})$ may be written as
\[
\xygraph{
!{<0cm,0cm>;<1cm,0cm>:<0cm,1cm>::}
!~-{@{-}@[|(2.5)]}
!{(0,1.5) }*+{\bullet}="x"
!{(1.5,3) }*+{\bullet}="a1"
!{(3,3) }*+{\bullet}="a2"
!{(1.5,0) }*+{\bullet}="c1"
!{(4.5,0) }*+{\bullet}="c2"
!{(1.5,1.5) }*+{\bullet}="b1"
!{(3,0) }*+{\dots}="cm"
!{(3,3) }*+{\bullet}="am"
!{(3,1.5) }*+{\dots}="bm"
!{(4.5,1.5) }*+{\bullet}="b2"
!{(1.5,1.9) }*+{2}
!{(4.5,1.9) }*+{2}
!{(1.5,3.4) }*+{q}
!{(3,3.4) }*+{k}
!{(0,1.9) }*+{2}
!{(1.5,0.4) }*+{2}
!{(4.5,0.4) }*+{2}
!{(3,1) }*+{\overbrace{\hphantom{--------}}^{kq^2-q-n-1}}
!{(3,2.5) }*+{\overbrace{\hphantom{--------}}^{q-1}}
"x"-"c1"
"x"-"a1"
"x"-"b1"
"a1"-"am"
"c1"-"cm"
"a2"-"am"
"b1"-"bm"
"b2"-"bm"
"c2"-"cm"
}
\]
This graph is positive definite.

If $k=2$ this graph does not embed, by \cref{2chain};
in fact, since $q\geq 3$, the above graph contains a $2$--chain whose 
length is at least $4$. As usual the vectors belonging
to this $2$--chain may be written as
\[
e_1+e_2,\dots,e_q+e_{q+1},\dots,e_{kq^2-n-1}+e_{kq^2-n},
\]
where $e_q+e_{q+1}$ corresponds to the central vertex of the graph. Note that we can only use one more
basis vector to write down the image of the vectors with weights $q$ and $k$. 
Let us denote these vectors as $v_1$ and $v_2$. Let us also put $N=kq^2-q-n$. 

If $k=3$ we must have $v_2\cdot e_i=0$ for each $i\leq N$.
It follows that $v_2=\alpha e_{N+1}$, which is impossible because $v_2\cdot v_2 = k = 3$ is not a square.

If $k=5$ either we may apply the same argument as above or the length of the $2$--chain is $4$. In this last case we 
quickly obtain $n=41$ which is not a square.

Similarly if $k=6$ we have two possibilities according to whether the length
of the $2$--chain is $4$ or $5$. If this length is $4$ then $q=3$ and $n=49$. These values correspond to the triple 
$(17,3;49)$. By \cite{55letters} the corresponding surgered manifold bounds a $\mathbb{Q}HB^4$. If the length of the $2$--chain
is $5$ then either $q=3$ and $n=48$ or $q=4$ and $n=94$ both cases can be excluded since $n$ is not a square.
We keep arguing in the same way. For a fixed value of $k$ (which is not a square) we consider the possible lengths of 
$2$--chain. Each length gives several possible pairs for $(q,n)$. We list the values obtained in this way below. Let
$l$ be the length of the $2$--chain.
\begin{itemize}
 \item if $q=3$ then $56\le n=62-l \le 58$;
 \item if $q=4$ then $105 \le n = 111-l \le 106$;
 \item if $q=5$ then $n = 174-l = 168$;
\end{itemize}
All these cases can be excluded because $n$ is not a square.

If $k=8$ we proceed exactly in the same way. 
In this case there are some candidate triples, but they can all be excluded by looking at the vector $v_1$ whose weight is $q$. We omit the details. Now suppose that $k\in\{4,9\}$.
In both cases we can exclude the expression $v_2=\sqrt{k}e_i$ for some $i\leq N+1$ (because $v_1\cdot v_2=1$).
Therefore we may proceed as we did for the other values of $k$.

Suppose $k=4$. The vector $v_2$ must hit some basis vector which appears in the expression of the $2$--chain, but then it must hit all of them. 
Since the $2$--chain has length at least $4$ this is impossible (because $v_2\cdot v_2=4$).

Suppose now that $k=9$. Just like above
we conclude that the length of the $2$--chain is at most $8$, this means that $4\leq 9q^2-n-1\leq 8$. Moreover
there is a portion of the $2$--chain whose length is $q-1$ which implies that $3\leq q\leq 7$. These two inequalities are
enough to spell out all possible pairs $(q,n)$ as indicated below.
\begin{itemize}
 \item if $q=3$ then $75\le n=80-l \le 79$;
 \item if $q=4$ then $139 \le n = 143-l \le 142$;
 \item if $q=5$ then $221 \le n = 224-l \le 223$;
 \item if $q=6$ then $321 \le n = 323-l \le 322$;
 \item if $q=7$ then $n=340-l=339$.
\end{itemize}
In no case is $n$ a square, and this concludes this subcase.
\vskip 0.5 cm
\textbf{Third subcase}: $n\in\{kq^2-q-1,kq^2-q,kq^2-q+1\}$.\\
 If $n=pq=kq^2-q$ we have 
\[
-S^3_n(T_{kq-1,q})=L(kq-1,q)\# L(q,kq-1)=L(kq-1,q)\# L(q,q-1).
\]
It follows from \cite{Lisca-ribbon} that if this manifold bounds a $\mathbb{Q}HB^4$ then $q=4$. In particular
both summands in the above expression must bound a $\mathbb{Q}HB^4$. The first summand is $L(4k-1,4)$. 
It does not bound because $4k-1$ is not square
. If $n=kq^2-q\pm 1$ we have
\[
-S^3_n(T_{kq-1,q})=L(kq^2-q\pm 1,q^2).
\]
We start with $L(kq^2-q+1,q^2)$. Using \cref{confrac} we obtain
\[
\frac{kq^2-q+1}{q^2}=k-\frac{1}{\frac{q^2}{q-1}}=[k,q+2,2^{[q-2]}]^-.
\]
It follows that $I(\frac{kq^2-q+1}{q^2})=k-2$. If $L(kq^2-q+1,q)$ bounds a $\mathbb{Q}HB^4$ then
$-3\leq I(\frac{kq^2-q+1}{q^2})\leq 1$ which implies $2\leq k\leq 3$. The linear plumbing graph of $-L(kq^2-q+1,q^2)$
can be depicted as

\[
\xygraph{
!{<0cm,0cm>;<1cm,0cm>:<0cm,1cm>::}
!~-{@{-}@[|(2.5)]}
!{(0,0) }*+{\bullet}="x"
!{(1.5,0) }*+{\bullet}="a1"
!{(3,0) }*+{\bullet}="a2"
!{(4.5,0) }*+{\dots}="c1"
!{(6,0) }*+{\bullet}="c2"
!{(0,0.4) }*+{k}
!{(1.5,0.4) }*+{q+2}
!{(3,0.4) }*+{2}
!{(6,0.4) }*+{2}
!{(4.5,1) }*+{\overbrace{\hphantom{--------}}^{q-2}}
"x"-"a1"
"a2"-"a1"
"a2"-"c1"
"c1"-"c2"
}.
\]
If $k=2$ and $q>3$, by 
\cref{2chain} this graph does not embed. If $k=2$ and $q=3$ the corresponding surgered manifold is $L(16,9)$
which, by \cite{Lisca-ribbon} bounds a $\mathbb{Q}HB^4$. The corresponding triple for this surgery is $(5,3;16)$.
Now we 
may assume that $k=3$. It can be checked directly that in this case the graph does not embed; start with the $2$--chain
and the vertex whose weight is 3.

Now we look at $L(kq^2-q-1,q^2)$. Using \cref{confrac} we obtain
\[
\frac{kq^2-q-1}{q^2}=k-\frac{1}{\frac{q^2}{q+1}}=[k,q,2^{[q]}]^-.
\]
It follows that $I(\frac{kq^2-q-1}{q^2})=k-6$, but since $-3\leq I(\frac{kq^2-q-1}{q^2})\leq 1$ we have $3\le k\le 7$. 
The linear plumbing graph of $-L(kq^2-q+1,q^2)$
can be depicted as

\[
\xygraph{
!{<0cm,0cm>;<1cm,0cm>:<0cm,1cm>::}
!~-{@{-}@[|(2.5)]}
!{(0,0) }*+{\bullet}="x"
!{(1.5,0) }*+{\bullet}="a1"
!{(3,0) }*+{\bullet}="a2"
!{(4.5,0) }*+{\dots}="c1"
!{(6,0) }*+{\bullet}="c2"
!{(0,0.4) }*+{k}
!{(1.5,0.4) }*+{q}
!{(3,0.4) }*+{2}
!{(6,0.4) }*+{2}
!{(4.5,1) }*+{\overbrace{\hphantom{--------}}^{q}}
"x"-"a1"
"a2"-"a1"
"a2"-"c1"
"c1"-"c2"
}
\]
As usual, the $2$--chain can be written as
$$
(e_1+e_2,e_2+e_3,\dots,e_q+e_{q+1}).
$$
Call $v$ the vector whose weight is $q$. It is easy to see that $v\cdot e_i=0$ if $2\leq i\leq q+1$ and therefore we may write $v=e_1+\sqrt{q-1} e_{q+2}$. Call $w$ the vector whose weight is $k$.
We can write 
$$
w=\gamma e_{q+2}+\beta\sum_{i=1}^{q+1}(-1)^i e_i.
$$
Since $w\cdot w=k$ and $w\cdot v=1$, we have $(q+1)^2\beta^2+\gamma^2=k$ and $\beta+\gamma\sqrt{q-1}=1$. From which we obtain
$$
(q+1)^2\beta^2+\frac{(1-\beta)^2}{q-1}=k.
$$
Since $q\geq 3$, $\beta\neq 0$ and $k\leq 7$ we get a contradiction and the proof is complete.
\end{proof}